\documentclass{article}
\usepackage{amsmath, amsfonts,amsthm, amssymb,url,hyperref}
\usepackage{graphics}
\usepackage[english]{babel}

\newcommand{\ba}{\begin{eqnarray}}
\newcommand{\ea}{\end{eqnarray}}

\newtheorem{theorem}{Theorem}
\newtheorem{assumption}{Assumption}

\newtheorem{proposition}{Proposition}

\newtheorem{remark}{Remark}[section]
\newtheorem{lemma}{Lemma}

\title{Mimicking the marginal distributions of a semimartingale}
 \date{September 2009. Revised: April 2012.\footnote{Laboratoire de Probabilit\'es et
    Mod\`eles Al\'eatoires, UMR 7599 CNRS - Universit\'e Pierre \& Marie Curie (Paris VI)}}
\author{Amel Bentata\  and Rama Cont}

\begin{document}

\maketitle
\begin{abstract}
We show that the flow of marginal distributions of a discontinuous
semimartingale $X$ can be matched by a Markov process whose
infinitesimal generator is expressed in terms of the local
characteristics of $X$. The  conditions under which such {\it
Markovian projections} exist are shown to hold for a large class of
stochastic processes used in applications. Our results extend  a
``mimicking theorem" of Gy\"ongy (1986)
  to discontinuous semimartingales.  We use this result to
 derive a partial integro-differential equation for the
one-dimensional distributions of  a semimartingale, extending the
Kolmogorov forward equation to a non-Markovian setting.
\end{abstract}
MSC Classification Numbers:     60J75,  60H10 \tableofcontents
\newpage

\section{Introduction}
Stochastic processes with path-dependent, non-Markovian dynamics are
used in various fields such as physics and mathematical finance
present challenges for computation, simulation and estimation. In
some applications  where one is interested in the marginal
distributions of such processes, such as option pricing or Monte
Carlo simulation of densities, the complexity of the model can be
greatly reduced by considering a low-dimensional Markovian model
with the same marginal distributions. Given a process $\xi$, a
Markov process $X$ is said to {\it mimick} $\xi$ on the time
interval $[0,T]$, $T>0$, if $\xi$ and $X$ have the same marginal
distributions:
\begin{equation}\label{mimic.eq} \forall t\in[0,T],\qquad
\xi_t\mathop{=}^{d} X_t.\end{equation} $X$ is called a {\it
Markovian projection} of $\xi$.
 The construction of
Markovian projections was first suggested by Br\'emaud
\cite{bremaud81} in the context of queues. Construction of mimicking
processes of 'Markovian' type has been explored for Ito processes
\cite{gyongy86} and marked point processes \cite{contminca08}. A
notable application is the derivation of forward equations for
option pricing \cite{forward,dupire97}.

We propose in this paper a systematic construction of such
Markovian projections for  (possibly discontinuous) semimartingales.
Given a semimartingale $\xi$, we give conditions under which there
exists a Markov process $X$ whose   marginal distributions are
identical to those of
  $\xi$, and give an explicit construction of the Markov process $X$ as the solution of a martingale problem for an integro-differential operator \cite{bass04,komatsu73,stroock75,stroock03}.

In the martingale case, the Markovian projection problem is related
to the problem of constructing martingales with a given flow of
marginals, which dates back to Kellerer \cite{kellerer} and has been
recently explored by Yor and coauthors
\cite{bakeryor09,hirschyor09,madanyor02} using a variety of
techniques. The construction proposed in this paper is different from the others since it does not rely on the martingale property
of $\xi$. We shall see nevertheless that our construction
preserves the (local) martingale property. Also, whereas the approaches
described in \cite{bakeryor09,hirschyor09,madanyor02} use as a
starting point the marginal distributions of $\xi$, our construction
describes the mimicking Markov process $X$ in terms of the local
characteristics \cite{jacodshiryaev} of the semimartingale $\xi$.
Our construction thus applies more readily to solutions of
stochastic differential equations where the local characteristics
are known but not the marginal distributions.

Section \ref{mimic.sec} presents a  Markovian projection result for
a $\mathbb{R}^d$-valued semimartingale given by its local
characteristics. We use these results in section \ref{fokkerplanck.sec}
to derive a partial integro-differential equation for the
one-dimensional distributions of a discontinuous semimartingale,
 thus extending the Kolmogorov forward equation to a non-Markovian
setting. Section \ref{extension.sec} shows how this result may be
applied to processes whose jumps are represented as the integral of
a predictable jump amplitude with respect to a Poisson random
measure, a representation often used in stochastic differential
equations with jumps. In Section \ref{examples.sec} we show that our
construction applies to a large class of semimartingales, including
smooth functions of a Markov process (Section \ref{markov.sec}),
and
time-changed L\'evy processes (Section \ref{timechange.sec}).

\section{A mimicking theorem for discontinuous semimartingales}\label{mimic.sec}
Consider, on a filtered probability space
$(\Omega,\mathcal{F},(\mathcal{F}_t)_{t\geq 0},\mathbb{P})$, an Ito
semimartingale, on the time interval $[0,T]$, $T>0$, given by the decomposition
\begin{equation}\label{classeJ}
  \xi_t=\xi_0+\int_0^t \beta_s\,ds+\int_0^t \delta_s\,dW_s
  +\int_0^t\int_{\|y\|\leq 1}y\,\tilde{M}(ds\:dy)+ \int_0^t\int_{\|y\|> 1}y\,{M}(ds\:dy),
\end{equation}
where $\xi_0$ is in $\mathbb{R}^d$, $W$ is a standard
$\mathbb{R}^n$-valued Wiener process, $M$ is an integer-valued
random measure  on
  $[0,T]\times\mathbb{R}^d$ with compensator measure $\mu$ and
$\tilde{M}=M-\mu$ is the compensated measure
\cite[Ch.II,Sec.1]{jacodshiryaev}, $\beta$ (resp. $\delta$) is an
 adapted process with values in $\mathbb{R}^d$ (resp. $M_{d\times n}(\mathbb{R})$).

Our goal   is to construct a Markov process, on some filtered
probability space $(\Omega_0,{\cal B},({\cal B}_t)_{t\geq
0},\mathbb{Q})$ such that $X$ and $\xi$ have the same marginal
distributions on $[0,T]$, i.e. the law of $X_t$ under $\mathbb{Q}$ coincides with the law of $\xi_t$ under $\mathbb{P}$.
We will construct $X$ as the solution to a {\it
martingale problem}
\cite{ethierkurtz,stroock75,stroockvaradhan,mikulevicius92} on the
canonical space $\Omega_0=D([0,T],\mathbb{R}^d)$.

\subsection{Martingale problems for integro-differential operators}
Let $\Omega_0=D([0,T],\mathbb{R}^d)$ be the Skorokhod space of right-continuous functions with left limits.
Denote by $X_t(\omega)=\omega(t)$ the canonical process on $\Omega_0$, ${\cal B}_t^0$ its natural filtration and
$\mathcal{B}_t\equiv\mathcal{B}_{t+}^0$.

Our goal   is to construct a probability measure $\mathbb{Q}$ on
$\Omega_0$ such that $X$ is a Markov process under $\mathbb{Q}$  and
$\xi$ and $X$ have the same one-dimensional distributions:
$$\forall t\in[0,T],\quad  \xi_t\overset{\underset{\mathrm{d}}{}}{=}  X_t.$$
In order to do this, we shall characterize $\mathbb{Q}$  as the solution of a {\it martingale problem}
 for an appropriately chosen integro-differential operator $\mathcal{L}$.

Let $\mathcal{C}_b^0(\mathbb{R}^d)$ denote the set of  bounded and
continuous functions on $\mathbb{R}^d$, $\mathcal{C}_0^\infty(\mathbb{R}^d)$ the set of  infinitely
differentiable functions with compact support on $\mathbb{R}^d$ and $\mathcal{C}_0(\mathbb{R}^d)$ the set of continuous functions defined on $\mathbb{R}^d$ and vanishing at infinity for the supremum norm. Let
$\mathcal{R}(\mathbb{R}^d-\{0\})$ denote the space of L\'evy
measures on $\mathbb{R}^d$ i.e. the set of non-negative
$\sigma$-finite measures $\nu$ on $\mathbb{R}^d-\{0\}$ such that
$$\int_{\mathbb{R}^d-\{0\}} \nu(dy) \left(1\wedge \|y\|^2\right) < \infty. $$
 $\mathcal{R}(\mathbb{R}^d-\{0\})$ is endowed with the structure of
 a measurable space, such that for each $\nu\in \mathcal{R}(\mathbb{R}^d-\{0\})$, the map
 \begin{eqnarray*}
 \mathcal{C}_b^0(\mathbb{R}^d)& \mapsto& \mathbb{R},\\[0.1cm]
 \varphi & \mapsto& \int_{\mathbb{R}^d-\{0\}} \nu(dy)\frac{
 \|y\|^2}{1+\|y\|^2} \varphi(y)
\end{eqnarray*}
is measurable.

Consider a time-dependent integro-differential operator
$\mathcal{L}=(\mathcal{L}_t)_{t\in[0,T]}$ defined, for  $f\in\mathcal{C}_0^\infty(\mathbb{R}^d)$, by
  \begin{equation}\label{non.deg.op}
   \begin{split}
      \mathcal{L}_tf(x)&=b(t,x).\nabla f(x)+\sum_{i,j=1}^d\frac{a_{ij}(t,x)}{2}\frac{\partial^2 f}{\partial x_i\partial x_j} (x)\\[0.1cm]
      &+\int_{\mathbb{R}^d}[f(x+y)-f(x)-1_{\{\|y\|\leq1\}}\,y.\nabla f(x)]n(t,dy,x),
   \end{split}
  \end{equation}
where $a:[0,T]\times \mathbb{R}^d\mapsto M_{d\times d}(\mathbb{R})$, $b:[0,T]\times \mathbb{R}^d\mapsto \mathbb{R}^d$ and \\
$n:[0,T]\times \mathbb{R}^d\mapsto \mathcal{R}(\mathbb{R}^d-\{0\})$
are measurable functions.

For $(t_0,x_0)\in[0,T]\times \mathbb{R}^d$, we recall that a probability measure
$\mathbb{Q}_{t_0,x_0}$ on $\left(\Omega_0,\mathcal{B}_T\right)$ is a solution to the {\it martingale problem} for
$(\mathcal{L},\mathcal{C}_0^\infty(\mathbb{R}^d))$ on $[0,T]$ if
$\mathbb{Q}\left(X_{u}=x_0,\,0\leq u\leq t_0\right)=1$ and for any
$f\in\mathcal{C}_0^\infty(\mathbb{R}^d)$, the process
\[ f(X_t)-f(x_0)-\int_{t_0}^t \mathcal{L}_sf(X_s)\,ds\]
is a $\left(\mathbb{Q}_{t_0,x_0},(\mathcal{B}_t)\right)$-martingale on $[0,T]$. Existence, uniqueness and
regularity of solutions to martingale problems for
integro-differential operators have been studied under various
conditions on the coefficients
\cite{stroockvaradhan,jacod79,ethierkurtz,komatsu73,mikulevicius92,figalli}.

We make the following assumptions on the coefficients:
\begin{assumption}[Boundedness  of coefficients]\label{A1}
\begin{eqnarray*}
(i)&\exists K_1>0,& \forall (t,z)\in[0,T]\times\mathbb{R}^d, \quad\|b(t,z)\| + \|a(t,z)\|+ \int \left(1\wedge \|y\|^2\right)\,n(t,dy,z)\leq K_1\\[0.1cm]
(ii)&&\lim_{R\to\infty} \int_0^T \sup_{z\in\mathbb{R}^d} n\left(t,\{\|y\|\geq R\},z\right)\,dt=0.
\end{eqnarray*}
where $\|.\|$ denotes the Euclidean norm.
\end{assumption}
\begin{assumption}[Continuity]\label{A2}\  \\
(i) For   $t\in [0,T]$ and $B\in\mathcal{B}(\mathbb{R}^d -\{0\}),$
$b(t,.)$, $a(t,.)$ and $n(t,B,.)$ are continuous on $\mathbb{R}^d$,
uniformly in $t\in [0,T]$. \\[0.1cm]
(ii) For all $z\in\mathbb{R}^d$, $b(.,z)$,
$a(.,z)$ and $n(.,B,z)$ are right-continuous on $[0,T[$, uniformly
in $z\in \mathbb{R}^d$.
\end{assumption}
\begin{assumption}[Non-degeneracy]\label{C}
\begin{eqnarray*}
  &\mathrm{Either} & \quad \forall R>0\:\forall t\in [0,T] \quad \inf_{\|z\| \leq R}\,\inf_{x\in\mathbb{R}^d,\,\|x\|=1} {}^tx.a(t,z).x>0 \\[0.1cm]
 & \mathrm{or}& a\equiv 0\quad{\rm and\ there\ exists}\quad \beta\in]0,2[, C>0,\,\mathrm{and\ a\
 family}\:
 n^\beta(t,dy,z)\\[0.1cm]
&& \mathrm{of \ positive\:measures\  such\ that}\\[0.1cm]
  &&\forall (t,z)\in [0,T]\times\mathbb{R}^d\quad \quad n(t,dy,z)=n^\beta(t,dy,z)+\frac{C}{\|y\|^{d+\beta}}\,dy,\\[0.1cm]
  && \int \left(1\wedge \|y\|^{\beta}\right)\,n^\beta(t,dy,z) \leq K_2,\quad \lim_{\epsilon\to 0} \sup_{z\in\mathbb{R}^d} \int_{\|y\|\leq \epsilon} \|y\|^{\beta}\,n^\beta(t,dy,z)=0. \\[0.1cm]
\end{eqnarray*}
\end{assumption}
Mikulevicius and Pragarauskas \cite{mikulevicius92}
show that
 if  $\mathcal{L}$ satisfies Assumptions \ref{A1},\ref{A2} and \ref{C} ( which corresponds to a ``non-degenerate L\'evy operator" in the terminology of \cite{mikulevicius92}) the martingale problem for $(\mathcal{L},\mathcal{C}_0^\infty(\mathbb{R}^d)\ )$ has a unique solution $\mathbb{Q}_{t_0,x_0}$ for every initial condition $(t_0,x_0)\in[0,T]\times\mathbb{R}^d$:
\begin{proposition}\label{th.well.posed}
 Under Assumptions \ref{A1}, \ref{A2} and \ref{C} the martingale problem for $((\mathcal{L}_t)_{t\in[0,T]},\mathcal{C}_0^\infty(\mathbb{R}^d))$  on $[0,T]$ is well-posed : for any
$x_0\in\mathbb{R}^d, t_0\in [0,T]$, there exists a unique
probability measure $\mathbb{Q}_{t_0,x_0}$ on
$(\Omega_0,{\cal B}_T)$ such that
$\mathbb{Q}\left(X_{u}=x_0,\,0\leq u\leq t_0\right)=1$ and for any
$f\in\mathcal{C}_0^\infty(\mathbb{R}^d)$,
\[ f(X_t)-f(x_0)-\int_{t_0}^t \mathcal{L}_sf(X_s)\,ds\]
is a $\left(\mathbb{Q}_{t_0,x_0},(\mathcal{B}_t)_{t\geq 0}\right)$-martingale on $[0,T]$.
Under $\mathbb{Q}_{t_0,x_0}$, $(X_t)$ is a  Markov process and the
evolution operator $(Q_{t_0,t})_{t\in[t_0,T]}$ defined by
\begin{equation}
\forall f\in \mathcal{C}_b^0(\mathbb{R}^d),\quad Q_{t_0,t}
f(x_0)=\mathbb{E}^{\mathbb{Q}_{t_0,x_0}}\left[f(X_t)\right]
\end{equation}
verifies the following continuity property:
\begin{equation}
\forall f\in \mathcal{C}_0^\infty(\mathbb{R}^d), \quad\lim_{t\downarrow t_0} Q_{t_0,t}f (x_0) =
 f(x_0).\label{strong.continuity}
\end{equation}
In particular, denoting $q_{t_0,t}(x_0,dy)$ the marginal distribution  of $X_t$, the map
\begin{equation}\label{weak-right-continuity}
 t\in[t_0,T[\mapsto \int_{\mathbb{R}^d} q_{t_0,t}(x_0,dy) f(y)
\end{equation}
is right-continuous, for any $f\in \mathcal{C}_0^\infty(\mathbb{R}^d)$.
\end{proposition}
\begin{proof}
By a result of Mikulevicius and Pragarauskas  \cite[Theorem 5]{mikulevicius92}, the martingale problem  is well-posed.
We only need to prove that the continuity property \eqref{strong.continuity} holds on $[t_0,T[$ for any $x_0\in\mathbb{R}^d$. 
For $f\in\mathcal{C}_0^\infty(\mathbb{R}^d)$,
\begin{eqnarray*}
 Q_{t_0,t} f(x_0)&=&\mathbb{E}^{\mathbb{Q}_{t_0,x_0}}\left[f(X_t)\right]\\[0.1cm]
&=&f(x_0)+\mathbb{E}^{\mathbb{Q}_{t_0,x_0}}\left[\int_{t_0}^t \mathcal{L}_sf(X_s)\,ds\right].\\
\end{eqnarray*}
Given Assumption \ref{A1}, $t\in[t_0,T]\mapsto \int_0^t \mathcal{L}_sf(X_s)\,ds$ is uniformly bounded on $[t_0,T]$. By Assumption \ref{A2}, since $X$ is right continuous, $s\in[t_0,T[\mapsto \mathcal{L}_sf(X_s)$ is right-continuous up to a $\mathbb{Q}_{t_0,x_0}$-null set and
\begin{equation*}
\lim_{t\downarrow t_0} \int_{t_0}^t \mathcal{L}_sf(X_s)\,ds= 0 \quad\mathrm{a.s.}
\end{equation*}
Applying the dominated convergence theorem yields,
\begin{equation*}
\lim_{t\downarrow t_0} \mathbb{E}^{\mathbb{Q}_{t_0,x_0}}\left[\int_{t_0}^t \mathcal{L}_sf(X_s)\,ds\right]=0,
\end{equation*}
that is
\begin{equation*}
\lim_{t\downarrow t_0} Q_{t_0,t} f(x_0) = f(x_0),
\end{equation*}
implying that $t\in[t_0,T[\mapsto Q_{t_0,t} f(x_0)$ is right-continuous at $t_0$.
\end{proof}

\subsection[Uniqueness for the Kolmogorov forward equation]{A uniqueness result for the Kolmogorov forward equation}
An important property of continuous-time Markov processes is their
link with partial (integro-)differential equation (PIDE) which
allows to use analytical tools for studying their probabilistic
properties. In particular the transition density of a Markov process
solves the forward Kolmogorov equation (or Fokker-Planck equation)
\cite{stroock03}.  The following result shows that under Assumptions \ref{A1}, \ref{A2} and \ref{C}  the   forward equation corresponding to ${\cal L}$ has a unique solution:
\begin{theorem}[Kolmogorov Forward equation]
\label{evolution.equation}
Under Assumptions \ref{A1}, \ref{A2} and \ref{C}, for each $(t_0,x_0)\in[0,T]\times\mathbb{R}^d$, there exists a unique family
$\left(p_{t_0,t}(x_0,dy),t\in[t_0,T]\right)$ of bounded measures on $\mathbb{R}^d$ such that $p_{t_0,t_0}(x_0,.)=\epsilon_{x_0}$, the point mass at $x_0$ and
\begin{equation}\label{kolmogore}
\forall t\in[t_0,T],\forall g\in  \mathcal{C}_0^\infty(\mathbb{R}^d),\qquad\int_{\mathbb{R}^d}
g(y)\frac{dp_{t_0,t}}{dt}(x_0,dy)=\int_{\mathbb{R}^d} p_{t_0,t}(x_0,dy)L_tg(y).\end{equation}
$p_{t_0,t}(x_0,.)$ is  the conditional distribution of $X_t$ given $X_{t_0}=x_0$, where $(X,\mathbb{Q}_{t_0,x_0})$ is the unique solution of the martingale problem for $(\mathcal{L},\mathcal{C}_0^\infty(\mathbb{R}^d)\ )$ starting from $(t_0,x_0)$.
\end{theorem}
\begin{proof}
\begin{enumerate}
\item Under Assumptions \ref{A1}, \ref{A2} and \ref{C},  Proposition \ref{th.well.posed} implies that the martingale problem for $\mathcal{L}$ on the domain $\mathcal{C}_0^\infty(\mathbb{R}^d)$ is well-posed.
Denote  $(X,\mathbb{Q}_{t_0,x_0})$ the unique solution of the martingale problem for $\mathcal{L}$ with initial condition  $x_0\in\mathbb{R}^d$ at $t_0$, and define
\begin{equation}
\forall t\geq t_0,\quad \forall g\in \mathcal{C}_b^0(\mathbb{R}^d),\quad Q_{t_0,t} g(x_0)=\mathbb{E}^{\mathbb{Q}_{t_0,x_0}}\left[g(X_t)\right].
\end{equation}
By \cite[Theorem 5]{mikulevicius92}, $( Q_{s,t}, 0\leq s\leq t)$ is then a (time-inhomogeneous) semigroup, satisfying the continuity property \eqref{strong.continuity} on $[t_0,T[$.

If $q_{t_0,t}(x_0,dy)$ denotes the law of $X_t$ under $\mathbb{Q}_{t_0,x_0}$, the martingale property implies that $q_{t_0,t}(x_0,dy)$ satisfies
\begin{equation}\label{eq.qt.bis}
\forall g\in\mathcal{C}_0^\infty(\mathbb{R}^d),\quad \int_{\mathbb{R}^d}  q_{t_0,t}(x_0,dy)g(y)=g(x_0)+\int_{t_0}^t\int_{\mathbb{R}^d} q_{t_0,s}(x_0,dy)\mathcal{L}_sg(y)\,ds.
\end{equation}
Proposition \ref{th.well.posed} provides the right-continuity of \\
$t\in[t_0,T[\mapsto \int_{\mathbb{R}^d}  q_{t_0,t}(x_0,dy)g(y)$ for any $g$ in $\mathcal{C}_0^\infty(\mathbb{R}^d)$. Given Assumption \ref{A2}, $q_{t_0,t}$ is a solution of \eqref{kolmogore} with initial condition $q_{t_0,t_0}(dy)=\epsilon_{x_0}$.
This solution of (\ref{kolmogore}) is in particular positive with mass $1$.

To show uniqueness of solutions of \eqref{kolmogore},  we will rewrite \eqref{kolmogore} as the forward Kolmogorov equation
associated with a {\it homogeneous}  operator on space-time domain and use uniqueness results for the corresponding homogeneous equation.

\item Let $\mathcal{D}^0\equiv \mathcal{C}^1([0,T])\otimes \mathcal{C}_0^\infty(\mathbb{R}^d)$ be the tensor product of $\mathcal{C}^1([0,T])$ and $\mathcal{C}_0^\infty(\mathbb{R}^d)$.
Define  the operator $A$ on $\mathcal{D}^0$ by 
\begin{equation}\label{new.operator}
\forall f\in C^\infty_0(\mathbb{R}^d),\forall \gamma\in\mathcal{C}^1([0,T]),\quad A(f\gamma)(t,x)=\gamma(t)\mathcal{L}_tf(x)+f(x)\gamma'(t).
\end{equation}
\cite[Theorem 7.1, Chapter 4]{ethierkurtz} implies that for any $x_0\in \mathbb{R}^d$, if  $(X,\mathbb{Q}_{t_0,x_0})$ is a solution of the
martingale problem for $\mathcal{L}$, then the law of $\eta_t=(t,X_t)$ under $\mathbb{Q}_{t_0,x_0}$ is a
solution of the martingale problem for  $A$: in particular for any
$f\in\mathcal{C}_0^\infty(\mathbb{R}^d)$ and
$\gamma\in\mathcal{C}^1([0,T])$,
\begin{equation}
     \int q_t(x_0,dy)f(y)\gamma(t)= f(x_0)\gamma(0)+\int_0^t \int q_s(x_0,dy)A(f\gamma)(s,y)\,ds.
    \end{equation}
 \cite[Theorem 7.1, Chapter 4]{ethierkurtz} implies also that if the law of $\eta_t=(t,X_t)$ is a
solution of the martingale problem for $A$ then the law of $X$ is also a solution of the
martingale problem for $\mathcal{L}$, namely: uniqueness holds for the martingale problem associated to the
operator $\mathcal{L}$ on $\mathcal{C}_0^\infty(\mathbb{R}^d)$ if and only if uniqueness
holds for the martingale problem associated to the  martingale problem for $A$ on $\mathcal{D}^0$.

Define, for  $t\in[0,T] $ and $h \in \mathcal{C}_b^0([0,T]\times\mathbb{R}^d)$,
\begin{equation}
\forall (s,x) \in[0,T] \times \mathbb{R}^d,\quad \mathcal{U}_{t}h(s,x)= Q_{s,s+t}\left( h(t+s,.) \right)(x).
\end{equation}
The properties of $Q_{s,t}$ then imply that $(\mathcal{U}_{t},t\geq 0)$ is a family of linear operators on   $\mathcal{C}_b^0([0,T]\times\mathbb{R}^d)$ satisfying $\mathcal{U}_{t}\mathcal{U}_{r}=\mathcal{U}_{t+r}$ on  $\mathcal{C}_b^0([0,T]\times\mathbb{R}^d)$ and $\mathcal{U}_{t}h \to h$ in as $t \downarrow 0$ on $\mathcal{D}^0$. $(\mathcal{U}_{t},t\geq 0)$ is thus a  semigroup on $\mathcal{C}_b^0([0,T]\times\mathbb{R}^d)$ satisfying the continuity property of the form \eqref{strong.continuity} on $\mathcal{D}^0$.

One observes that, for $h\in\mathcal{D}^0$,
\begin{equation*}
 \mathcal{U}_{t}h(s,x)= Q_{s,s+t}\left( h(t+s,.) \right)(x) = \int_{\mathbb{R}^d} q_{s,s+t} (x,dy) h(t+s,y).
\end{equation*}
Without loss of generality, let us put $t_0=0$ in the sequel. Since the martingale problem holds for $\eta_t=(t,X_t)$, then for all $h\in\mathcal{D}^0$, the martingale property yields,
\begin{equation}\label{u.mart}
\begin{split}
&\forall t\in [0,T],\forall (s,x)\in [0,T]\times\mathbb{R}^d \\
&\quad\mathcal{U}_{t} h(s,x)= \mathcal{U}_0 h(s,x) +\int_0^t \mathcal{U}_u Ah(s,x)\,du.
\end{split}
\end{equation}
Considering again this equality for $0\leq \epsilon <t$,
\begin{equation}\label{u.eps.mart}
 \mathcal{U}_t h- \mathcal{U}_{\epsilon} h=\int_{\epsilon}^t \mathcal{U}_u Ah\,du.
\end{equation}
Denoting $\mathcal{C}_0([0,T]\times\mathbb{R}^d)$ the set of continuous functions defined on $[0,T]\times\mathbb{R}^d$ and vanishing at infinity for the supremum norm,  we intend to apply \cite[Theorem 2.2, Chapter 4]{ethierkurtz} to prove that $(\mathcal{U}_{t},t\geq 0)$ generates a strongly continuous contraction on $\mathcal{C}_0([0,T]\times\mathbb{R}^d)$ with infinitesimal generator given by the closure $\overline{A}$. First, one shall simply observe that $\mathcal{D}^0$ is dense in $\mathcal{C}_0([0,T]\times\mathbb{R}^d)$ implying that the domain of $A$ is dense in $\mathcal{C}_0([0,T]\times\mathbb{R}^d)$ too. The well-posedness of the martingale problem for $A$ implies that $A$ satisfies the maximum principle. To conclude, it is sufficient to prove that $Im(\lambda-\overline{A})$ is dense in $\mathcal{C}_0([0,T]\times\mathbb{R}^d)$ for some $\lambda$ or even better that $\mathcal{D}^0$ is included in $Im(\lambda-\overline{A})$. We recall that $Im(\lambda-\overline{A})$ denotes the image of the domain of $\overline{A}$ under the map $(\lambda-\overline{A})$.

\item Consider $0\leq \epsilon<T$, (\ref{u.eps.mart}) yields for $h\in\mathcal{D}^0$,
\begin{eqnarray*}
\int_{\epsilon}^T e^{-t} \mathcal{U}_t h\,dt &=& \int_{\epsilon}^T e^{-t} \mathcal{U}_{\epsilon} h\,dt +\int_{\epsilon}^T e^{-t} \int_{\epsilon}^t \mathcal{U}_s Ah\,ds\,dt\\
&=& \mathcal{U}_{\epsilon} h\left[e^{-\epsilon}-e^{-T}\right] + \int_{\epsilon}^T ds\, \left(\int_{s}^T e^{-t}\,dt\right)\,\mathcal{U}_s Ah\\
&=& \mathcal{U}_{\epsilon} h\left[e^{-\epsilon}-e^{-T}\right] + \int_{\epsilon}^T ds\, \left[ e^{-s}- e^{-T}\right]\,\mathcal{U}_s Ah\\
&=&e^{-\epsilon}\,\mathcal{U}_{\epsilon} h-e^{-T}\left[\mathcal{U}_{\epsilon} h +  \int_{\epsilon}^T \mathcal{U}_s Ah\,ds\right]\\
&+&\int_{\epsilon}^T ds\,e^{-s}\,\mathcal{U}_s Ah.
\end{eqnarray*}
Using (\ref{u.mart}) and gathering all the terms together yields,
\begin{equation}\label{2.x}
\begin{split}
\int_{\epsilon}^T e^{-t} \mathcal{U}_t h\,dt
&=e^{-\epsilon}\,\mathcal{U}_{\epsilon} h-e^{-T} \mathcal{U}_T h\\
&+\int_{\epsilon}^T ds\,e^{-s}\,\mathcal{U}_s Ah.
\end{split}
\end{equation}
Let us focus on the quantity
\begin{equation*}
\int_{\epsilon}^T ds\,e^{-s}\,\mathcal{U}_s Ah.
\end{equation*}
Observing that,
\begin{equation*}
\frac{1}{\epsilon}\left[\mathcal{U}_{t+\epsilon} h-\mathcal{U}_{t} h\right]
= \frac{1}{\epsilon} \left[\mathcal{U}_{\epsilon}-I\right]\mathcal{U}_t h
=\mathcal{U}_t\,\frac{1}{\epsilon} \left[\mathcal{U}_{\epsilon}-I\right]h,
\end{equation*}
taking $\epsilon\to 0$ yields
\begin{equation*}
\mathcal{U}_{t}  \frac{1}{\epsilon} \left[\mathcal{U}_{\epsilon}-I\right]h\to U_{t} Ah
\end{equation*}
Hence, the limit when $\epsilon\to 0$ of
\begin{equation*}
\frac{1}{\epsilon} \left[\mathcal{U}_{\epsilon}-I\right]\mathcal{U}_t h
\end{equation*}
exists, implying that $\mathcal{U}_th $ belongs to the domain of $A$ for any $h\in\mathcal{D}^0$. Thus,
\begin{equation*}
\int_{\epsilon}^T ds\,e^{-s}\,\mathcal{U}_sh
\end{equation*}
belongs to the domain of $\overline{A}$ and
\begin{equation*}
\int_{\epsilon}^T ds\,e^{-s}\,\mathcal{U}_s Ah=\overline{A}\int_{\epsilon}^T ds\,e^{-s}\,\mathcal{U}_s h.
\end{equation*}


Since $\mathcal{U}$ is a contraction semigroup, thus a contraction, and given the right-continuity property of $\mathcal{U}_t$ on the space $\mathcal{D}^0$, one may take $\epsilon\to 0$ and $T\to \infty$ in (\ref{2.x}), leading to
\begin{equation*}
\int_{0}^\infty e^{-t} \mathcal{U}_t h\,dt
=\mathcal{U}_{0} h(0,x_0)+\overline{A}\,\int_{0}^{\infty} ds\,e^{-s}\,\mathcal{U}_s h.
\end{equation*}
Thus,
\begin{equation*}
\left(I-\overline{A}\right)\,\int_{0}^{\infty} ds\,e^{-s}\,\mathcal{U}_s h(0,x_0)= \mathcal{U}_{0} h= h,
\end{equation*}
yielding $h\in Im(I-\overline{A})$. We have shown that $(\mathcal{U}_{t},t\geq 0)$ generates a strongly continuous contraction on $\mathcal{C}_0([0,T]\times\mathbb{R}^d)$ with infinitesimal generator $\overline{A}$ (see \cite[Theorem 2.2, Chapter 4]{ethierkurtz}).
\item  The Hille-Yosida theorem (see \cite[Proposition 2.6, Chapter 1]{ethierkurtz}) then implies  that for all $\lambda>0$
$$Im(\lambda-\overline{A})=\mathcal{C}_0([0,T]\times\mathbb{R}^d).$$

\item Now consider for $t\geq 0, h \in \mathcal{C}_b^0([0,T]\times\mathbb{R}^d)$,
\begin{equation}
 \mathcal{Q}_{t}h(x_0)= \int_{\mathbb{R}^d} \,q_{t}(x_0,dy)h(t,y) =  \left( \mathcal{U}_{t}h   \right) (0,x_0).
\end{equation}
Using  (\ref{eq.qt.bis}), we have, for  $\epsilon>0$,
 \begin{eqnarray}\label{semigroup}
\forall (f,\gamma)\in\mathcal{C}^1([0,T])\times \mathcal{C}_0^\infty(\mathbb{R}^d),\qquad
 \mathcal{Q}_{t}(f\gamma)(x_0)-\mathcal{Q}_{\epsilon}(f\gamma)(x_0) =\nonumber\\[0.1cm]
\int_{\epsilon}^{t} \int_{\mathbb{R}^d} q_u(x_0,dy)A(f\gamma)(u,y)\,du
=\int_{\epsilon}^{t} \mathcal{Q}_{u}(A(f\gamma))(x_0)\,du .
\end{eqnarray}
By linearity,  for any $h\in\mathcal{D}^0$ we have
\begin{equation}
 \mathcal{Q}_{t}h(x_0)-\mathcal{Q}_{\epsilon}h(x_0) =\int_{\epsilon}^{t} \int_{\mathbb{R}^d} q_u(x_0,dy)Ah(u,y)\,du=\int_{\epsilon}^{t} \mathcal{Q}_{u}Ah(x_0)\,du,
\end{equation}
Now let $p_t(x_0,dy)$ be another  solution of (\ref{kolmogore}) such that $p_0(x_0,dy)=\epsilon_{x_0}(dy)$. Then $p_t$ is also a solution of (\ref{eq.qt.bis}). An integration  by parts implies that, for $ (f,\gamma)\in\mathcal{C}^1([0,T])\times \mathcal{C}_0^\infty(\mathbb{R}^d),$
    \begin{equation}\label{eq.pt.bis}
      \int_{\mathbb{R}^d} p_t(x_0,dy)f(y)\gamma(t)= f(x_0)\gamma(0)+\int_0^t \int_{\mathbb{R}^d} p_s(x_0,dy)A(f\gamma)(s,y)\,ds.
    \end{equation}
Define, for $h\in\mathcal{C}_b^0([0,T]\times\mathbb{R}^d)$,
\begin{eqnarray*}
\forall (t,x_0)\in [0,T]\times \mathbb{R}^d,\qquad \mathcal{P}_{t} h(x_0)&=& \int_{\mathbb{R}^d} p_t(x_0,dy)h(t,y).
\end{eqnarray*}
Using (\ref{eq.pt.bis}) we have, for $(f,\gamma)\in\mathcal{C}^1([0,T])\times \mathcal{C}_0^\infty(\mathbb{R}^d)$,
\begin{equation}\label{semigroup.bis}
\forall \epsilon>0\quad \mathcal{P}_{t}(f\gamma)-\mathcal{P}_{\epsilon}(f\gamma) =\int_{\epsilon}^{t} \int_{\mathbb{R}^d} p_u(dy)A(f\gamma)(u,y)\,du=\int_{\epsilon}^{t} \mathcal{P}_{u}(A(f\gamma))\,du.
\end{equation}
which is identical to \eqref{semigroup}. Multiplying by $e^{-\lambda t}$ and integrating with respect to $t$ we obtain that, for $\lambda>0,$
\begin{eqnarray*}
\lambda \int_0^\infty e^{-\lambda t} \,\mathcal{P}_{t}(f\gamma)(x_0)\,dt&=&f(x_0)\gamma(0)+\lambda\int_0^\infty e^{-\lambda t}\int_0^t \mathcal{P}_{u}(A(f\gamma))(x_0)\,du\,dt\\[0.1cm]
&=&f(x_0)\gamma(0)+\lambda\int_0^\infty \left(\int_u^\infty e^{-\lambda t} dt\right)\, \mathcal{P}_{u}(A(f\gamma))(x_0)\,du\\[0.1cm]
&=&f(x_0)\gamma(0)+\int_0^\infty e^{-\lambda u}\, \mathcal{P}_{u}(A(f\gamma))(x_0)\,du.
\end{eqnarray*}
Similarly, from  \eqref{semigroup} we obtain for any $\lambda>0,$
\begin{eqnarray*}
\lambda \int_0^\infty e^{-\lambda t} \,\mathcal{Q}_{t}(f\gamma)(x_0)\,dt
&=&f(x_0)\gamma(0)+\int_0^\infty e^{-\lambda u}\, \mathcal{Q}_{u}(A(f\gamma))(x_0)\,du.
\end{eqnarray*}
Hence for $(f,\gamma)\in\mathcal{C}^1([0,T])\times \mathcal{C}_0^\infty(\mathbb{R}^d)$ we have
\begin{equation}\label{identity.semigroup}
\int_0^\infty e^{-\lambda t}\,\mathcal{Q}_{t}(\lambda-A)(f\gamma)(x_0)\,dt=f(x_0)\gamma(0)=\int_0^\infty e^{-\lambda t}\,\mathcal{P}_{t}(\lambda- A)(f\gamma)\,dt.
\end{equation}
By linearity,  for any $h\in\mathcal{D}^0$ we have
\begin{equation}
\int_0^\infty e^{-\lambda t}\,\mathcal{Q}_{t}(\lambda-A)h(x_0)\,dt=h(0,x_0)=\int_0^\infty e^{-\lambda t}\,\mathcal{P}_{t}(\lambda- A)h(x_0)\,dt
\end{equation}
Using the density of $\mathcal{D}^0$ in $\mathcal{C}_0([0,T]\times\mathbb{R}^d)$, we conclude that, for $h\in\mathcal{C}_0([0,T]\times\mathbb{R}^d)$
\begin{equation}\label{identity.semigroup.bis}
\int_0^\infty e^{-\lambda t}\,\mathcal{Q}_{t}(\lambda-\overline{A})h(x_0)\,dt=\int_0^\infty e^{-\lambda t}\,\mathcal{P}_{t}(\lambda- \overline{A})h(x_0)\,dt
\end{equation}
Finally, using the fact that $$Im(\lambda-\overline{A})=\mathcal{C}_0([0,T]\times\mathbb{R}^d),$$
we conclude that
\begin{equation}\label{identity.semigroup.tris}
\forall h\in \mathcal{C}_0([0,T]\times\mathbb{R}^d), \quad\int_0^\infty e^{-\lambda t}\,\mathcal{Q}_{t}h\,(x_0)\,dt=\int_0^\infty e^{-\lambda t}\,\mathcal{P}_{t}h(x_0) \,dt,
\end{equation}
so the Laplace transform of $t\mapsto \mathcal{P}_{t}h\,(x_0)$ is uniquely determined.


Using \eqref{semigroup.bis},
\begin{eqnarray}\label{semigroup.tris}
\forall \epsilon>0, \forall h\in\mathcal{D}^0,\nonumber\\[0.1cm]
 \mathcal{P}_{t}h-\mathcal{P}_{\epsilon}h =\int_{\epsilon}^{t} \int_{\mathbb{R}^d} p_u(dy)Ah(u,y)\,du=\int_{\epsilon}^{t} \mathcal{P}_{u}(Ah)\,du
\end{eqnarray}
by linearity, which allows to show that, for any $h\in\mathcal{D}^0,$
$ t\mapsto \mathcal{P}_{t}h(x_0)$ is right-continuous:
$$
\forall h\in \mathcal{D}^0,\qquad \lim_{t' \downarrow t} \mathcal{P}_{t'}h(x_0)=\mathcal{P}_th(x_0).
$$
An identical argument using \eqref{semigroup.bis} shows that $
t\mapsto \mathcal{Q}_{t}h(x_0)$ is right-continuous. These two
right-continuous functions have the same Laplace transform by
\eqref{identity.semigroup.tris}, so they are equal. Thus we have
shown that
\begin{equation}
\forall  h\in\mathcal{D}^0,\quad \int h(t,y) q_{t}(x_0,dy)=
\int h(t,y) p_t(x_0,dy). \label{eq.uniqueness}
\end{equation}
Since $\mathcal{D}^0$ is dense in
$\mathcal{C}_0([0,T]\times\mathbb{R}^d)$  for the supremum norm,
\eqref{eq.uniqueness} also holds for $h\in
\mathcal{C}_0([0,T]\times\mathbb{R}^d)$. By \cite[Proposition 4.4,
Chapter 3]{ethierkurtz}, $\mathcal{C}_0([0,T]\times\mathbb{R}^d)$ is
convergence determining, hence separating, allowing us to conclude
that $p_t(x_0,dy)=q_t(x_0,dy)$.
\end{enumerate}
\end{proof}

\begin{remark}
Assumptions \ref{A1}, \ref{A2} and \ref{C} are sufficient but not necessary for the well-posedness of the martingale problem. For example, the boundedness Assumption \ref{A1} may be relaxed to  local boundedness, using localization techniques developed in \cite{stroock75,stroockvaradhan}. Such extensions are not trivial and, in the unbounded case, additional conditions are needed to ensure that $X$ does not explode (see \cite[Chapter 10]{stroockvaradhan}).
\end{remark}

\subsection{Markovian projection of a semimartingale}\label{markovianproj.sec}

The following assumptions on the local characteristics of the semimartingale $\xi$ are almost-sure analogs of Assumptions \ref{A1}, \ref{A2} and \ref{C}:
\begin{assumption}\label{H1} $\beta,\delta$ are  bounded on $[0,T]$:
\begin{equation*}
 \exists K_1>0, \forall t\in[0,T],\:\: \|\beta_t\|\leq K_1,\:\: \|\delta_t\|\leq K_1\qquad \mathrm{a.s.}
\end{equation*}
\end{assumption}
\begin{assumption}\label{H2} The jump compensator
 $\mu$ has a density $m(\omega,t,dy)$ with respect to the
Lebesgue measure on $[0,T]$ which satisfies
\begin{eqnarray*}
(i)&\quad&\exists K_2>0, \forall t\in[0,T]\quad \int_{\mathbb{R}^d} \left(1\wedge \|y\|^2\right)\,m(.,t,dy)\leq K_2 <\infty\qquad\mathrm{a.s.}\\[0.1cm]
(ii)&\mathrm{and}&\lim_{R\to\infty} \int_0^T m\left(.,t,\{\|y\|\geq R\}\right)\,dt=0 \quad\mathrm{a.s.}
\end{eqnarray*}
\end{assumption}
\begin{assumption}\label{H3}
\begin{eqnarray*}
  &\mathrm{Either}& (i)\quad\exists \epsilon>0,\forall t\in [0,T[\,\, {}^t\delta_t\delta_t\geq \epsilon\ I_d \quad\mathrm{a.s.} \label{elliptic.eq}\\[0.1cm]
  &{\rm or} & (ii) \quad \delta\equiv 0\quad{\rm and\ there\ exists\,}\beta\in]0,2[, c,K_3>0,\,\mathrm{and\ a\
 family}\,m^\beta(t,dy)\\[0.1cm]
&& \mathrm{of \ positive\:measures\  such\ that}\\[0.1cm]
  &&\forall t\in [0,T[\quad \quad m(t,dy)=m^\beta(t,dy)+\frac{c}{\|y\|^{d+\beta}}\,dy\:\mathrm{a.s.},  \\[0.1cm]
  &&\quad \int \left(1\wedge \|y\|^{\beta}\right)\,m^\beta(t,dy) \leq K_3,\quad {\rm and}\quad \lim_{\epsilon\to 0} \int_{\|y\|\leq \epsilon} \|y\|^{\beta}\,m^\beta(t,dy)=0\:\mathrm{a.s.}
\end{eqnarray*}
\end{assumption}
Note that Assumption \ref{H2} is only slightly stronger than stating that $m$ is
a L\'evy kernel since in that case we already have   $\int \left(1\wedge
\|y\|^2\right)\,m(.,t,dy) <\infty$. Assumption \ref{H3} extends the ``ellipticity" assumption to the case of pure-jump semimartingales and holds for a large class of semimartingales driven by stable or tempered stable processes.
\begin{theorem}[Markovian projection]\label{th1}
Assume there exists measurable functions $a:[0,T]\times \mathbb{R}^d\mapsto M_{d\times d}(\mathbb{R})$, $b:[0,T]\times \mathbb{R}^d\mapsto \mathbb{R}^d$ and \\
$n:[0,T]\times \mathbb{R}^d\mapsto \mathcal{R}(\mathbb{R}^d-\{0\})$ satisfying Assumption \ref{A2} such that for all $t\in[0,T]$ and $B\in\mathcal{B}(\mathbb{R}^d-\{0\}),$
  \begin{equation}
    \begin{split}
      &\mathbb{E}\left[\beta_t|\xi_{t^-}\right]=b(t,\xi_{t-})\quad\mathrm{a.s},\\[0.1cm]
      &\mathbb{E}\left[{}^t\delta_t\delta_t|\xi_{t^-}\right]=a(t,\xi_{t-})\quad\mathrm{a.s},\\[0.1cm]
      &\mathbb{E}\left[m(.,t,B)|\xi_{t^-}\right]=n(t,B,\xi_{t-})\quad\mathrm{a.s}.
    \end{split}
  \end{equation}
If $(\beta,\delta,m)$ satisfies Assumptions  \ref{H1}, \ref{H2}, \ref{H3}, then  there exists a Markov process $((X_t)_{t\in[0,T]},\mathbb{Q}_{\xi_0})$, with infinitesimal
generator $\mathcal{L}$ defined by \eqref{non.deg.op}, whose   marginal distributions  mimick those of $\xi$:
  $$\forall t\in [0,T], \quad X_t\overset{\underset{\mathrm{d}}{}}{=}\xi_t. $$
  $X$ is the  weak solution  of the stochastic
differential equation
\begin{equation} \label{y.sde}
  \begin{split}
    X_t&=\xi_0+\int_0^t b(u,X_u)\,du+\int_0^t \Sigma(u,X_u)\,dB_u\\[0.1cm]
    &+\int_0^t\int_{\|y\|\leq 1}y \,\tilde{N}(du\ dy)+ \int_0^t\int_{\|y\|> 1}y \,N(du\ dy),
  \end{split}
  \end{equation}
where  $(B_t)$ is an n-dimensional Brownian motion, $N$ is an integer-valued
  random measure on $[0,T]\times\mathbb{R}^d$ with compensator
  $n(t,dy,X_{t-})\,dt$, $\tilde{N}=N-n$ the associated
compensated random measure and $\Sigma\in C^0([0,T]\times\mathbb{R}^d, M_{d\times n}(\mathbb{R}))$
 such that
$ {}^t\Sigma(t,z)\Sigma(t,z) = a(t,z) .$
\end{theorem}
We will call $(X,\mathbb{Q}_{\xi_0})$ the {\it Markovian projection} of $\xi$.
\begin{proof}
First, we observe that $n$ is a L\'evy kernel : for any $(t,z)\in[0,T]\times\mathbb{R}^d$
    \begin{eqnarray*}
      \int_{\mathbb{R}^d} \left(1\wedge \|y\|^2\right)\,n(t,dy,z)=
 \mathbb{E}\left[ \int_{\mathbb{R}^d} \left(1\wedge \|y\|^2\right)\,m(t,dy)|\xi_{t^-}=z\right]\ <\infty\:\:\mathrm{a.s.},
    \end{eqnarray*}
using Fubini's theorem and Assumption \ref{H2}.
Consider now the case of a pure jump semimartingale verifying (ii) and define, for $B\in\mathcal{B}(\mathbb{R}^d -\{0\})$,
\begin{equation*}
\forall z\in\mathbb{R}^d\quad n^{\beta}(t,B,z)=\mathbb{E}\left[\int_{B} m(t,dy,\omega)-\frac{c\,dy}{\|y\|^{d+\beta}}|\xi_{t^-}=z\right].
\end{equation*}
As argued above, $n^\beta$ is a L\'{e}vy kernel on $\mathbb{R}^d$.
Assumptions \ref{H1} and \ref{H2} imply that $(b,a,n)$ satisfies Assumption \ref{A1}. Furthermore, under assumptions either $(i)$ or $(ii)$ for $(\delta,m)$, Assumption \ref{C} holds for $(b,a,n)$. Together with Assumption \ref{A2} yields that $\mathcal{L}$ is a non-degenerate operator and Proposition \ref{th.well.posed} implies that the
martingale problem for $(\mathcal{L}_t)_{t\in [0,T]}$ on the domain $\mathcal{C}_0^\infty(\mathbb{R}^d)$ is well-posed.
Denote  $((X_t)_{t\in [0,T]},\mathbb{Q}_{\xi_0})$ its unique solution starting from $\xi_0$ and $q_t(\xi_0,dy)$ the marginal distribution of $X_t$.
Let $f\in \mathcal{C}_0^\infty(\mathbb{R}^d)$. It\^{o}'s formula
yields
    \begin{eqnarray*}
      f(\xi_t)&=&f(\xi_0)+\sum_{i=1}^d\int_0^t \sum_{i=1}^d \frac{\partial f}{\partial x_i}(\xi_{s^-})\,d\xi_s^i+\frac{1}{2}\int_0^t
      {\rm tr}\left[\nabla^2 f(\xi_{s^-})\,{}^t\delta_s\delta_s\right]\,ds\\[0.1cm]
      &+&\sum_{s\leq t} \left[f(\xi_{s^-}+\Delta
        \xi_s)-f(\xi_{s^-})-\sum_{i=1}^d \frac{\partial f}{\partial x_i}(\xi_{s^-})\Delta  \xi_s^i \right]\\[0.1cm]
      &=&f(\xi_0)+\int_0^t \nabla f(\xi_{s^-}).\beta_{s}\,ds+
  \int_0^t \nabla f(\xi_{s^-}).\delta_{s}dW_s\\[0.1cm]
  &+&\frac{1}{2}\int_0^t {\rm tr}\left[\nabla^2 f(\xi_{s^-})\,{}^t\delta_s\delta_s\right]\ \,ds +\int_0^t\int_{\|y\|\leq 1}\nabla f(\xi_{s^-}).y\,\tilde{M}(ds\,dy)\\[0.1cm]
  &+&\int_0^t\int_{\mathbb{R}^d}\left(f(\xi_{s^-}+y)-f(\xi_{s^-})-1_{\{\|y\|\leq1\}}\,y.\nabla
    f(\xi_{s^-})\right)\,M(ds\,dy).
    \end{eqnarray*}
 We note that
\begin{itemize}
\item since $\left\|\nabla f\right\|$ is
bounded  $\int_0^t\int_{\|y\|\leq 1}\nabla f(\xi_{s^-}).y\,\tilde{M}(ds\,dy)$ is  a square-integrable martingale.
\item  $\int_0^t\int_{\|y\|> 1}\nabla
  f(\xi_{s^-}).y\,{M}(ds\,dy)< \infty\, \mathrm{a.s.}$
since $\left\|\nabla f\right\|$ is bounded.
\item since $\nabla f(\xi_{s^-})$  and $\delta_s$ are uniformly bounded on $[0,T]$, $\int_0^t \nabla f(\xi_{s^-}).\delta_{s} dW_s$ is a martingale.
\end{itemize}
Hence, taking  expectations, we obtain:
\begin{eqnarray*}
  \mathbb{E}^{\mathbb{P}}\left[f(\xi_t)\right]&=&\mathbb{E}^{\mathbb{P}}\left[f(\xi_0)\right]+\mathbb{E}^{\mathbb{P}}\left[\int_0^t \nabla f(\xi_{s^-}).\beta_{s}\,ds\right]+ \mathbb{E}^{\mathbb{P}}\left[\frac{1}{2}\int_0^t {\rm tr}\left[\nabla^2 f(\xi_{s^-}){}^t\delta_s\delta_s\right]\,ds\right]\\[0.1cm]
  &+&\mathbb{E}^{\mathbb{P}}\left[\int_0^t\int_{\mathbb{R}^d}\left(f(\xi_{s^-}+y)-f(\xi_{s^-})-1_{\{\|y\|\leq1\}}\,y.\nabla
    f(\xi_{s^-})\right)\,M(ds\,dy)\right]\\[0.1cm]
&=&\mathbb{E}^{\mathbb{P}}\left[f(\xi_0)\right]+\mathbb{E}^{\mathbb{P}}\left[\int_0^t \nabla
  f(\xi_{s^-}).\beta_{s}\,ds\right]+\mathbb{E}^{\mathbb{P}}\left[\frac{1}{2}\int_0^t {\rm tr}\left[\nabla^2
  f(\xi_{s^-})\,{}^t\delta_s\delta_s\right]\,ds\right]\\[0.1cm]
&+&\mathbb{E}^{\mathbb{P}}\left[\int_0^t\int_{\mathbb{R}^d}\left(f(\xi_{s^-}+y)-f(\xi_{s^-})-1_{\{\|y\|\leq1\}}\,y.\nabla
    f(\xi_{s^-})\right)\,m(s,dy)\,ds\right].
\end{eqnarray*}
Observing that:
\begin{eqnarray*}
  &&\mathbb{E}^{\mathbb{P}}\left[\int_0^t \nabla f(\xi_{s^-}).\beta_{s}\,ds\right]\leq
  \|\nabla f\| \,\mathbb{E}^{\mathbb{P}}\left[\int_0^t\|\beta_s\|\,ds\right]<\infty,\\[0.1cm]
  &&\mathbb{E}^{\mathbb{P}}\left[\frac{1}{2}\int_0^t {\rm tr}\left[\nabla^2 f(\xi_{s^-})\,{}^t\delta_s\delta_s\right]\right]\leq \|\nabla^2f\|\, \mathbb{E}^{\mathbb{P}}\left[\int_0^t
    \|\delta_s\|^2\,ds\right]<\infty, \\[0.1cm]
  &&\mathbb{E}^{\mathbb{P}}\left[ \int_0^t\int_{\mathbb{R}^d} \left\| f(\xi_{s^-}+y)-f(\xi_{s^-})-1_{\{\|y\|\leq1\}}\,y.\nabla
    f(\xi_{s^-})\right\|\,m(s,dy)\,ds\right]\\[0.1cm]
      &&\quad \leq\frac{\|\nabla^2f\|}{2}\mathbb{E}^{\mathbb{P}}\left[\int_0^t\int_{\|y\|\leq 1}\|y\|^2\,m(s,dy)\,ds\right]+ 2 \|f\|\mathbb{E}^{\mathbb{P}}\left[\int_0^t\int_{\|y\|> 1}\,m(s,dy)\,ds\right]<+\infty,
    \end{eqnarray*}
        we may apply Fubini's theorem to obtain
        \begin{eqnarray*}
          \mathbb{E}^{\mathbb{P}}\left[f(\xi_t)\right]&=&\mathbb{E}^{\mathbb{P}}\left[f(\xi_0)\right]+\int_0^t \mathbb{E}^{\mathbb{P}}\left[\nabla
            f(\xi_{s^-}).\beta_{s}\right]\,ds + \frac{1}{2}\int_0^t \mathbb{E}^{\mathbb{P}}\left[{\rm tr}\left[\nabla^2
            f(\xi_{s^-})\,{}^t\delta_s\delta_s\right]\right] \,ds\\[0.1cm]
          &+&\int_0^t\mathbb{E}^{\mathbb{P}}\left[\int_{\mathbb{R}^d}\left(f(\xi_{s^-}+y)-f(\xi_{s^-})-1_{\{\|y\|\leq1\}}\,y.\nabla
              f(\xi_{s^-})\right)\,m(s,dy)\right]\,ds.
        \end{eqnarray*}
    Conditioning on $\xi_{t-}$ and using the iterated expectation
    property,
    \begin{eqnarray*}
      \mathbb{E}^{\mathbb{P}}\left[f(\xi_t)\right]&=&\mathbb{E}^{\mathbb{P}}\left[f(\xi_0)\right]+\int_0^t \mathbb{E}^{\mathbb{P}}\left[\nabla
        f(\xi_{s^-}). \mathbb{E}^{\mathbb{P}}\left[\beta_{s}|\xi_{s-}\right]\right]\,ds \\[0.1cm]
&+&\frac{1}{2}\int_0^t \mathbb{E}^{\mathbb{P}}\left[{\rm tr}\left[\nabla^2
            f(\xi_{s^-})\, \mathbb{E}^{\mathbb{P}}\left[{}^t\delta_s\delta_s|\xi_{s-}\right]\right]]\right] \,ds\\[0.1cm]
      &+&\int_0^t\mathbb{E}^{\mathbb{P}}\left[\mathbb{E}^{\mathbb{P}}\left[\int_{\mathbb{R}^d}\left(f(\xi_{s^-}+y)-f(\xi_{s^-})-1_{\{\|y\|\leq1\}}\,y.\nabla
          f(\xi_{s^-})\right)\,m(s,dy)|\xi_{s-}\right]\right]\,ds\\[0.1cm]
      &=&\mathbb{E}^{\mathbb{P}}\left[f(\xi_0)\right]+\int_0^t \mathbb{E}^{\mathbb{P}}\left[\nabla
        f(\xi_{s^-}).b(s,\xi_{s-})\right]\,ds
      +\frac{1}{2}\int_0^t \mathbb{E}^{\mathbb{P}}\left[{\rm tr}\left[\nabla^2
          f(\xi_{s^-})\,a(s,\xi_{s-})\right]\right] \,ds\\[0.1cm]
      &+&\int_0^t\int_{\mathbb{R}^d}\mathbb{E}^{\mathbb{P}}\left[\left(f(\xi_{s^-}+y)-f(\xi_{s^-})-1_{\{\|y\|\leq1\}}\,y.\nabla
        f(\xi_{s^-})\right)\,n(s,dy,\xi_{s-})\right]\,ds.
      \end{eqnarray*}
    Hence
    \begin{equation}\label{eq.pt}
      \mathbb{E}^{\mathbb{P}}\left[f(\xi_t)\right]=\mathbb{E}^{\mathbb{P}}\left[f(\xi_0)\right]+\mathbb{E}^{\mathbb{P}}\left[\int_0^t \mathcal{L}_sf(\xi_{s-})\,ds\right].
    \end{equation}
Let $p_t(dy)$ denote the law of $(\xi_t)$ under $\mathbb{P}$, (\ref{eq.pt}) writes:
    \begin{equation}
      \int_{\mathbb{R}^d} p_t(dy)f(y)= \int_{\mathbb{R}^d} p_0(dy)f(y)+\int_0^t \int_{\mathbb{R}^d} p_s(dy)\mathcal{L}_s f(y)\,ds.
    \end{equation}
Hence $p_t(dy)$ satisfies the Kolmogorov forward equation (\ref{kolmogore}) for the operator $\mathcal{L}$ with the initial condition $p_0(dy)=\mu_0(dy)$ where $\mu_0$ denotes the law of $\xi_0$.
Applying Theorem \ref{evolution.equation}, the flows $q_{t}(\xi_0,dy)$ of $X_t$ and $p_t(dy)$ of $\xi_t$ are the same on $[0,T]$. This ends the proof.

\end{proof}
\begin{remark}[Mimicking conditional distributions]
The construction in Theorem \ref{th1} may also be carried out using
  \begin{equation*}
    \begin{split}
      \mathbb{E}\left[\beta_t|\xi_{t^-},\,\mathcal{F}_0\right]&=b_0(t,\xi_{t-})\:\mathrm{a.s},\\[0.1cm]
      \mathbb{E}\left[{}^t\delta_t\delta_t|\xi_{t^-},\,\mathcal{F}_0\right]&=a_0(t,\xi_{t-})\:\mathrm{a.s},\\[0.1cm]
      \mathbb{E}\left[m(.,t,B)|\xi_{t^-},\,\mathcal{F}_0\right]&=n_0(t,B,\xi_{t-})\:\mathrm{a.s},
    \end{split}
  \end{equation*}
 instead of $(b,a,n)$ in \eqref{new_param}. If $(b_0,a_0,n_0)$ satisfies Assumption \eqref{C},
then following the same procedure we can construct a Markov process $(X,\mathbb{Q}^0_{\xi_0})$ whose infinitesimal generator has coefficients  $(b_0,a_0,n_0)$ such that
$$ \forall f\in\mathcal{C}_b^0(\mathbb{R}^d), \forall t\in[0,T]\quad \mathbb{E}^{\mathbb{P}}\left[f(\xi_t)|\mathcal{F}_0\right]=\mathbb{E}^{\mathbb{Q}^0_{\xi_0}}\left[f(X_t)\right] ,$$
i.e. the marginal distribution of $X_t$ matches the  conditional distribution of $\xi_t$ given $\mathcal{F}_0$.
\end{remark}
\begin{remark} For Ito processes (i.e. continuous semimartingales of the form \eqref{classeJ} with $\mu=0$),
 Gy\"ongy \cite[Theorem 4.6]{gyongy86} gives a ``mimicking theorem"
under the non-degeneracy condition $ {}^t \delta_t.\delta_t\geq \epsilon I_d$ which corresponds to our Assumption \ref{H3}, but without requiring the continuity condition (Assumption \ref{A2}) on $(b,a,n)$.
Brunick \& Shreve \cite{brunick2010} extend this result by relaxing the  ellipticity condition of \cite{gyongy86}.
In both cases, the mimicking process $X$ is constructed as a weak solution to the SDE \eqref{y.sde} (without the jump term), but this weak solution does {\it not} in general have the Markov property: indeed, it need not even be unique under the assumptions used in \cite{gyongy86,brunick2010}. In particular, in the setting used in \cite{gyongy86,brunick2010}, the law of $X$ is not uniquely determined by its 'infinitesimal generator' ${\cal L}$. This makes it difficult to `compute' quantities involving $X$, either through simulation or by solving a partial differential equation.

By contrast, under the additional continuity condition \ref{A2} on the projected coefficients,  $X$ is  a Markov process whose law is uniquely determined by its infinitesimal generator ${\cal L}$ and whose marginals are the unique solution of the Kolmogorov forward equation \eqref{kolmogore}.
This makes it possible to compute the marginals of $X$ by simulating the SDE \eqref{y.sde} or by solving a forward PIDE.

It remains to be seen whether the additional Assumption \ref{A2}  is verified in most examples of interest. We will show in Section \ref{examples.sec} that this is indeed the case.
\end{remark}

\begin{remark}[Markovian projection of a Markov process]
The term {\rm Markovian projection} is justified by the following remark: if the semimartingale $\xi$
is already a Markov process and satisfies the assumption of Theorem \ref{th1}, then the uniqueness in law of the solution to the martingale problem for ${\cal L}$ implies that the  Markovian projection $(X,\mathbb{Q}_{\xi_0})$ of $\xi$ has the same law as $(\xi,\mathbb{P}_{\xi_0})$. So the map which associates (the law $\mathbb{Q}_{\xi_0}$ of) $X$ to $\xi$ may indeed be viewed as a {\it projection}; in particular it is involutive.

This property contrasts with other constructions of
 mimicking processes
\cite{bakeryor09,contminca08,gyongy86,hamza06,madanyor02} which fail to be involutive. A striking example is the construction, by Hamza \& Klebaner \cite{hamza06}, of discontinuous martingales whose marginals match those of a Gaussian Markov process.
\end{remark}

\subsection{Forward equations for semimartingales}\label{fokkerplanck.sec}
Theorem \ref{evolution.equation} and Theorem \ref{th1} allow us to obtain a  forward PIDE which extends the Kolmogorov forward equation to   semimartingales which verify the Assumptions of Theorem \ref{th1}:
\begin{theorem}\label{forward.thm}
Let $\xi$ be a semimartingale given by \eqref{classeJ} satisfying the assumptions of Theorem \ref{th1}.
Denote
$p_{t}(dx)$  the law of $\,\xi_t$ on $\mathbb{R}^d$. The  $(p_{t})_{t\in [0,T]}$ is the unique solution, in the sense of
distributions, of the forward equation
\begin{equation}
\forall t\in[0,T],\quad\frac{\partial p_{t}}{\partial
  t}= \mathcal{L}^{\star}_t.\,p_{t},
\end{equation}
with initial condition $p_{0}=\mu_0$, where $\mu_0$ denotes the law of $\xi_0$,\\
where $\mathcal{L}^{\star}$ is the adjoint  of $\mathcal{L}$, defined by
\begin{eqnarray}\forall g &\in&C_0^\infty(\mathbb{R}^d,\mathbb{R}),\nonumber\\[0.1cm]
  \mathcal{L}^{\star}_t g(x) &=&-\nabla\left[b(t,x)g(x)\right]+\nabla^2\left[\frac{a(t,x)}{2}g(x)\right]\label{adjoint.eq}\\[0.1cm]
&+& \int_{\mathbb{R}^d}\left[g(x-z)n(t,z,x-z)-g(x)n(t,z,x)-1_{\|z\|\leq 1}z.\nabla\left[g(x)n(t,dz,x)\right]\right],\nonumber
\end{eqnarray}
where the coefficients $b,a,n$ are defined as in \eqref{new_param}.
\end{theorem}
\begin{proof}
The existence and uniqueness is a direct consequence of Theorem \ref{evolution.equation} and Theorem \ref{th1}. To finish the proof, let compute $\mathcal{L}_t^\star$.
Viewing $p_{t}$ as an element of the dual of $C^\infty_0(\mathbb{R}^d)$, (\ref{kolmogore}) rewrites : for $f \in C^\infty_0(\mathbb{R}^d,\mathbb{R})$
\begin{equation*}
\forall f\in  \mathcal{C}_0^\infty(\mathbb{R}^d,\mathbb{R}),\qquad\int
f(y)\frac{dp}{dt}(dy)=\int p_t(dy)\mathcal{L}_tf(y).
\end{equation*}
We have
$$ \forall f \in C^\infty_0(\mathbb{R}^d), \forall t\leq t'<T\quad <\frac{p_{t'}-p_t}{t'-t},f>\mathop{\to}^{t'\to t}<p_t,\mathcal{L}_tf>= <\mathcal{L}_t^*p_t,f>,$$
where $<.,.>$ is the duality product.

For $z\in \mathbb{R}^d$, define the translation operator $\tau^z$  by $\tau_zf(x)=f(x+z)$. Then
\begin{eqnarray*}
&&\int p_{t}(dx) \,\mathcal{L}_tf(x)\\[0.1cm]
&=&\int p_{t}(dx)\,\Big[b(t,x)\nabla f(x)+\frac{1}{2}{\rm tr}\left[\nabla^2 f(x)\,a(t,x)\right]\\[0.1cm]
&&\quad\quad\quad\quad+ \int_{|z|>1}(\tau_zf(x)-f(x))n(t,dz,x)\\[0.1cm]
&&\quad\quad\quad\quad+ \int_{|z|\leq1}(\tau_zf(x)-f(x)-z.\nabla f(x))\,n(t,dz,x)\Big]\\[0.1cm]
&=&\int \Big[-f(x)\frac{\partial }{\partial
  x}[b(t,x)p_{t}(dx)]+f(x)\frac{\partial^2 }{\partial
  x^2}[\frac{a(t,x)}{2}p_{t}(dx)]\\[0.1cm]
&&\quad\quad\quad\quad+ \int_{|z|>1} f(x)(\tau_{-z}(p_{t}(dx)n(t,dz,x))-p_{t}(dx)n(t,dz,x))\\[0.1cm]
&&\quad\quad\quad\quad+ \int_{|z|\leq1}
f(x)(\tau_{-z}(p_{t}(dx)n(t,dz,x))-p_{t}(dx)n(t,dz,x))\\[0.1cm]
&&\quad\quad\quad\quad-z\frac{\partial }{\partial
  x}(p_{t}(dx)n(t,dz,x))\Big],
\end{eqnarray*}
allowing to identify $\mathcal{L}^{\star}$.
\end{proof}

\subsection{Martingale-preserving property}
An important property of the construction of $\xi$ in Theorem
\ref{th1} is that it preserves the (local) martingale property: if $\xi$ is a local martingale, so is $X$:
\begin{proposition}[Martingale preserving property]\label{martingalepres.prop}\ \\[0.1cm]
\begin{enumerate}\item If $\xi$ is a {\rm local martingale}  which satisfies the assumptions of
Theorem \ref{th1},
then its Markovian projection $(X_t)_{t\in[0,T]}$ is  a local
martingale on $(\Omega_0,{\cal B}_t,\mathbb{Q}_{\xi_0})$.\item If furthermore
$$\mathbb{E}^{\mathbb{P}}\left[\int_0^T \int_{\mathbb{R}^d}\|y\|^2 m(t,dy) dt \right]<\infty,$$
then $(X_t)_{t\in[0,T]}$ is a square-integrable martingale.\end{enumerate}
\end{proposition}
\begin{proof}
1) If  $\xi$ is a local martingale then the uniqueness of its semimartingale decomposition
entails that
$$
\beta_t+\int_{\|y\|\geq 1} y \,m(t,dy)=0\,\qquad dt \times \mathbb{P}-a.e.
$$
$${\rm hence}\qquad
\mathbb{Q}_{\xi_0}\left(\forall t\in[0,T],\quad\int_0^t ds \,\left[b(s,X_{s-})+\int_{\|y\|\geq 1} y\, n(s,dy,X_{s-}) \right] = 0\right)\ =1.
$$
The assumptions on $m,\delta$ then entail that $X$, as a sum of
an Ito integral and a compensated Poisson integral, is a local
 martingale.\\[0.1cm]
2) If $\mathbb{E}^{\mathbb{P}}\left[\int \|y\|^2 \mu(dt\,dy)\right]<\infty$ then
$$
\mathbb{E}^{\mathbb{Q}_{\xi_0}}\left[\int \|y\|^2 n(t,dy,X_{t-})\right]<\infty,
$$
and the compensated Poisson integral in $X$ is a square-integrable martingale.
\end{proof}
\section[The case of semimartingales driven by  Poisson random measure]{Mimicking a semimartingale driven by a Poisson random measure}\label{extension.sec}
The representation \eqref{classeJ} is not the most commonly used in
applications, where a process is constructed as the solution to a
stochastic differential equation driven by a Brownian motion and a
Poisson random measure
\begin{equation}\label{classeK} 
 \zeta_t=\zeta_0+\int_0^t \beta_s\,ds+\int_0^t \delta_s\,dW_s+\int_0^t\int\psi_s(y)\,\tilde{N}(ds\:dy), 
\end{equation}
where $\xi_0\in\mathbb{R}^d$, $W$ is a standard
$\mathbb{R}^n$-valued Wiener process, $\beta$ and $\delta$ are non-anticipative c\`{a}dl\`{a}g
processes, $N$ is a Poisson random measure on
$[0,T]\times\mathbb{R}^d$ with intensity   $\nu(dy)\,dt$ where
\begin{equation}
\int_{\mathbb{R}^d} \left(1\wedge \|y\|^2\right) \nu(dy)<\infty,\qquad \tilde{N}=N-\nu(dy) dt,
\end{equation}
and the random jump amplitude $\psi:[0,T]\times\Omega\times
\mathbb{R}^d\mapsto \mathbb{R}^d$ is $\mathcal{P}\otimes {\cal
B}(\mathbb{R}^d)$-measurable, where
  $\mathcal{P}$ is the predictable $\sigma$-algebra on $[0,T]\times\Omega$.
In this section, we shall assume that
\begin{equation*}
\forall t\in[0,T],\quad\psi_t(\omega,0)=0\quad{\rm and}\quad\mathbb{E}\left[\int_0^t \int_{\mathbb{R}^d} \left(1\wedge \|\psi_s(.,y)\|^2\right)\,\nu(dy)\,ds\right]<\infty.
\end{equation*}
The difference between this representation and \eqref{classeJ}
is the presence of a random jump amplitude $\psi_t(\omega,.)$ in \eqref{classeK}.
The relation between these two representations for semimartingales has been discussed in great generality in \cite{karoui77,kabanov81}. Here we give a
less general result  which suffices for our purpose.
The following result expresses $\zeta$ in the form \eqref{classeJ} suitable
for applying Theorem \ref{th1}.
\begin{lemma}[Absorbing the jump amplitude in the compensator]\label{lem}
$$  \zeta_t=\zeta_0+\int_0^t \beta_s\,ds+\int_0^t \delta_s\,dW_s+\int_0^t\int\ \psi_s(z)\,\tilde{N}(ds\:dz)$$
can be also represented as
\begin{equation}\label{Z.eq}
  \zeta_t=\zeta_0+\int_0^t \beta_s\,ds+\int_0^t \delta_s\,dW_s+\int_0^t\int y\,\tilde{M}(ds\:dy),
\end{equation}
   where $M$ is an integer-valued random measure on  $[0,T]\times\mathbb{R}^d$
   with compensator $\mu(\omega,dt,dy) $ given by
 \begin{equation*}
 \forall A\in {\cal B}(\mathbb{R}^d-\{0\}),\qquad \mu(\omega,dt,A)=\nu( \psi^{-1}_t(\omega, A))\,dt,
 \end{equation*}
where $\psi^{-1}_t(\omega, A)=\{z\in \mathbb{R}^d, \psi_t(\omega, z)\in A \}$ denotes the inverse image of $A$ under the partial map $\psi_t$.
 \end{lemma}
 \begin{proof}
The result can be deduced from \cite[Th\'eor\`eme 12]{karoui77}  but we sketch here the proof for completeness.
 A Poisson random measure $N$ on $[0,T]\times\mathbb{R}^d$ can be represented as a counting measure for some random sequence $(T_n,U_n)$ with values in $[0,T]\times\mathbb{R}^d$
    \begin{equation}\label{defN}
      N=\sum_{n\geq 1}1_{\{T_n,U_n\}}.
    \end{equation}
 Let $M$ be the integer-valued random measure defined by:
    \begin{equation}\label{defM}
      M=\sum_{n\geq 1}1_{\{T_n,\psi_{T_n}(.,U_n)\}}.
    \end{equation}
 $\mu$,  the {\it predictable} compensator of $M$
     is characterized by the following property \cite[Thm 1.8.]{jacodshiryaev}: for
any  positive $\mathcal{P}\otimes {\cal B}(\mathbb{R}^d)$-measurable map $\chi:[0,T]\times\Omega\times\mathbb{R}^d \to\mathbb{R}^+$ and any $ A\in{\cal B}(\mathbb{R}^d-\{0\})$,
    \begin{equation}\label{compM}
      \mathbb{E}\left[\int_0^t\int_{A} \chi(s,y)\,M(ds\,dy)\right]=\mathbb{E}\left[\int_0^t\int_{A} \chi(s,y)\,\mu(ds\,dy)\right].
    \end{equation}
    Similarly, for $B \in {\cal B}(\mathbb{R}^d-\{0\})$
    \begin{equation*}
      \mathbb{E}\left[\int_0^t\int_{B} \chi(s,y)\,N(ds\,dy)\right]=\mathbb{E}\left[\int_0^t\int_{B} \chi(s,y)\,\nu(dy)\,ds\right].
    \end{equation*}
 Using formulae (\ref{defN}) and (\ref{defM}):
    \begin{eqnarray*}
      \mathbb{E}\left[\int_0^t\int_{A} \chi(s,y)\,M(ds\,dy)\right]&=&
      \mathbb{E}\left[\sum_{n\geq 1}\chi(T_n,\psi_{T_n}(., U_n))\right]\\[0.1cm]
      &=&\mathbb{E}\left[\int_0^t\int_{\psi^{-1}_s(., A)} \chi(s,\psi_s(.,z))\,N(ds\,dz)\right]\\[0.1cm]
      &=&\mathbb{E}\left[\int_0^t\int_{\psi^{-1}_s(., A)} \chi(s,\psi_s(.,z))\,\nu(dz)\,ds\right]
    \end{eqnarray*}
    Formula (\ref{compM}) and the above equalities lead to:
\begin{equation*}
\mathbb{E}\left[\int_0^t\int_{A} \chi(s,y)\,\mu(ds\,dy)\right]=\mathbb{E}\left[\int_0^t\int_{\psi^{-1}_s(., A)} \chi(s,\psi_s(.,z))\,\nu(dz)\,ds\right].
\end{equation*}
      Since $\psi$ is a predictable random function,  the uniqueness of the predictable compensator $\mu$ (take $\phi\equiv
      Id$ in
    \cite[Thm 1.8.]{jacodshiryaev}) entails
      \begin{equation}\label{chgtvariable}
        \mu(\omega,dt,A)=\nu( \psi^{-1}_t(\omega, A)) \,dt.
      \end{equation}
   Formula (\ref{chgtvariable}) defines a random measure $\mu$ which is a L\'{e}vy kernel 
       \begin{equation*}
       \int_0^t \int \left(1\wedge \|y\|^2\right)\,\mu(dy\,ds)=\int_0^t \int \left(1\wedge \|\psi_s(.,y)\|^2\right)\,\nu(dy)\,ds<\infty.
       \end{equation*}
 \end{proof}
In the case where  $\psi_t(\omega, .):  \mathbb{R}^d\mapsto \mathbb{R}^d $
is invertible and differentiable, we can characterize the density of the compensator $\mu$ as follows:
\begin{lemma}[Differentiable case]
If the L\'evy measure $\nu(dz)$ has a density $\nu(z)$ and if $\psi_t(\omega, .):  \mathbb{R}^d\mapsto \mathbb{R}^d $ is a $\mathcal{C}^1(\mathbb{R}^d,\mathbb{R}^d)$-diffeomorphism,
then  $\zeta$, given in \eqref{classeK}, has the representation
\begin{equation*}
  \zeta_t=\zeta_0+\int_0^t \beta_s\,ds+\int_0^t \delta_s\,dW_s+\int_0^t\int y\,\tilde{M}(ds\:dy),
\end{equation*}
   where $M$ is an integer-valued random measure
   with compensator
$$ m(\omega;t,y)\,dt\,dy = 1_{\psi_t(\omega,\mathbb{R}^d)}(y)\left|{\rm det}\nabla_y \psi_t\right|^{-1}(\omega,\psi_{t}^{-1}(\omega,y))\,\nu(\psi_t^{-1}(\omega,y))\,dt\,dy, $$
where $\nabla_y\psi_t$ denotes the Jacobian matrix of  $\psi_t(\omega, .)$. \label{lemma2}
\end{lemma}
\begin{proof}
We recall from the proof of Lemma \ref{lem}:
\begin{equation*}
\mathbb{E}\left[\int_0^t\int_{A} \chi(s,y)\,\mu(ds\,dy)\right]=\mathbb{E}\left[\int_0^t\int_{\psi^{-1}_s(., A)} \chi(s,\psi_s(.,z))\,\nu(z)\,ds\,dz\right].
\end{equation*}
Proceeding to the change of variable $\psi_s(.,z)=y$:
\begin{eqnarray*}
&&\mathbb{E}\left[\int_0^t\int_{\psi^{-1}_s(., A)} \chi(s,\psi_s(.,z))\,\nu(z)\,ds\,dz\right]\\[0.1cm]
&=&\mathbb{E}\left[\int_0^t\int_{A}1_{\{\psi_s(\mathbb{R}^d)\}}(y)\, \chi(s,y)\,\left|{\rm det}\nabla \psi_s\right|^{-1}(.,\psi_s^{-1}(.,y))\, \nu(\psi_s^{-1}(.,y)) ds\,dy\right].
\end{eqnarray*}
The density appearing in the right hand side is predictable  since $\psi$ is a predictable random function. The uniqueness of the predictable compensator $\mu$ yields the result.
\end{proof}
Let us combine Lemma \ref{lemma2} and Theorem \ref{th1}. To proceed, we make a further assumption.
\begin{assumption}\label{H3b}
The L\'evy measure  $\nu$ admits a density
  $\nu(y)$ with respect to the Lebesgue measure on $\mathbb{R}^d$ and:
\begin{eqnarray*}
&\quad&\forall t\in[0,T] \:\exists K_2>0\quad \int_0^t\int_{\|y\|>1} \left(1\wedge \|\psi_s(.,y)\|^2\right)\,\nu(y)\,dy\,ds<K_2 \:\:\mathrm{a.s.}\\[0.1cm]
&\quad\quad\mathrm{and}&\nonumber\\[0.1cm]
&\quad&\lim_{R\to\infty} \int_0^T \nu\left(\psi_t\left(\{\|y\|\geq R\}\right)\right)\,dt=0 \quad\mathrm{a.s.}
\end{eqnarray*}
\end{assumption}
\begin{theorem}\label{th2}
Let $(\zeta_t)$ be an Ito semimartingale defined on $[0,T]$ by the given the decomposition
\begin{equation*} 
  \zeta_t=\zeta_0+\int_0^t \beta_s\,ds+\int_0^t \delta_s\,dW_s+\int_0^t\int\psi_s(y)\,\tilde{N}(ds\:dy), 
\end{equation*}
where $\psi_t(\omega, .):  \mathbb{R}^d\mapsto \mathbb{R}^d $
is invertible and differentiable with inverse $\phi_t(\omega,.)$.
Define
\begin{equation}\label{mesure.image.de.phi}
m(t,y)= 1_{\{y\in\psi_t(\mathbb{R}^d)\}}\left|{\rm det}\nabla \psi_t\right|^{-1}(\psi_t^{-1}(y))\,\nu(\psi_t^{-1}(y)).
\end{equation}
Assume there exists measurable functions $a:[0,T]\times \mathbb{R}^d\mapsto M_{d\times d}(\mathbb{R}), b:[0,T]\times \mathbb{R}^d\mapsto \mathbb{R}^d$ and $j:(t,x)\in [0,T]\times \mathbb{R}^d)\to \mathcal{R}(\mathbb{R}^d-\{0\})$ satisfying Assumption \ref{A2} such that for $(t,z)\in[0,T]\times\mathbb{R}^d,B\in\mathcal{B}(\mathbb{R}^d-\{0\}),$
  \begin{equation}\label{new_param}
    \begin{split}
      \mathbb{E}\left[\beta_t|\zeta_{t^-}\right]&=b(t,\zeta_{t-})\:\mathrm{a.s},\\[0.1cm]
      \mathbb{E}\left[{}^t\delta_t\delta_t|\zeta_{t^-}\right]&=a(t,\zeta_{t-})\:\mathrm{a.s},\\[0.1cm]
      \mathbb{E}\left[m(.,t,B)|\zeta_{t^-}\right]&=j(t,B,\zeta_{t-})\:\mathrm{a.s}.
    \end{split}
  \end{equation}
If $\beta$ and $\delta$ satisfy Assumption \ref{H1}, $\nu$ Assumption \ref{H3b}, $(\delta,m )$ satisfy Assumptions \ref{H2}-\ref{H3}, then  the stochastic
differential equation
\begin{equation}
  X_t=\zeta_0+\int_0^t b(u,X_u)\,du+\int_0^t \Sigma(u,X_u)\,dB_u+\int_0^t\int y \,\tilde{J}(du\ dy),
  \end{equation}
where  $(B_t)$ is an $n$-dimensional Brownian motion, $J$ is an integer valued random measure on $[0,T]\times\mathbb{R}^d$ with compensator
  $j(t,dy,X_{t-})\,dt$, $\tilde{J}=J-j$  and $\Sigma:[0,T]\times\mathbb{R}^d\mapsto M_{d\times n}(\mathbb{R})$
  is a continuous function such that
$ {}^t\Sigma(t,z)\Sigma(t,z) = a(t,z)$, admits a unique weak solution  $((X_t)_{t\in[0,T]},\mathbb{Q}_{\zeta_0})$ whose marginal distributions mimick those of $\zeta$:
  $$\forall t\in[0,T]\quad  X_t\overset{\underset{\mathrm{d}}{}}{=}\zeta_t. $$
Under $\mathbb{Q}_{\zeta_0}$, $X$ is a  Markov process with infinitesimal
generator $\mathcal{L}$ given by \eqref{non.deg.op}.
\end{theorem}
\begin{proof}
We first use Lemma \ref{lemma2} to obtain the representation (\ref{Z.eq}) of $\zeta$:
\begin{equation*}
  \zeta_t=\zeta_0+\int_0^t \beta_s\,ds+\int_0^t \delta_s\,dW_s+\int_0^t\int y\,\tilde{M}(ds\:dy)
\end{equation*}
Then, we observe that
\begin{eqnarray*}
\int_0^t\int y\,\tilde{M}(ds\:dy)&=&\int_0^t\int_{\|y\|\leq 1} y\,\tilde{M}(ds\:dy)+\int_0^t\int_{\|y\|>1} y\,[M(ds\:dy)-\mu(ds\,dy)]\\[0.1cm]
&=&\int_0^t\int_{\|y\|\leq 1} y\,\tilde{M}(ds\:dy)+\int_0^t\int_{\|y\|>1} y\,M(ds\:dy)-\int_0^t\int_{\|y\|>1} y\,\mu(ds\,dy),
\end{eqnarray*}
where the terms above are well-defined thanks to Assumption \ref{H3b}.
Lemma \ref{lemma2} leads to:
\begin{equation*}
\int_0^t\int_{\|y\|>1} y\,\mu(ds\,dy)=\int_0^t\int_{\|\psi_s(y)\|>1} \|\psi_s(.,y)\|^2\,\nu(y)\,dy\,ds.
\end{equation*}
Hence:
\begin{equation*}
  \begin{split}
  \zeta_t&=\zeta_0+\left[\int_0^t \beta_s\,ds-\int_0^t\int_{\|\psi_s(y)\|>1} \|\psi_s(.,y)\|^2\,\nu(y)\,dy\,ds\right]+ \int_0^t \delta_s\,dW_s\\[0.1cm]
  &+\int_0^t\int_{\|y\|\leq 1} y\,\tilde{M}(ds\:dy)+\int_0^t\int_{\|y\|>1} y\,M(ds\:dy).
  \end{split}
\end{equation*}
This representation has the form \eqref{classeJ} and Assumptions \ref{H1} and \ref{H3b} guarantee that the  local characteristics  of $\zeta$ satisfy the assumptions of Theorem \ref{th1}. Applying Theorem \ref{th1} yields the result.
\end{proof}

\section{Examples}\label{examples.sec}
We now give some examples of stochastic models used in applications, where  Markovian projections
can be characterized in a more explicit manner than in the general results above. These examples also serve to illustrate that the continuity assumption (Assumption \ref{A2}) on the projected coefficients $(b,a,n)$ in \eqref{new_param} can be verified in many useful settings.
\subsection{Semimartingales  driven by a Markov process}\label{markov.sec}
In many examples in stochastic modeling, a quantity $Z$ is expressed as a smooth function
$f: \mathbb{R}^d\to\mathbb{R}$ of  a d-dimensional Markov process $Z$:
$$\xi_t=f(Z_t)\qquad {\rm with}\quad f: \mathbb{R}^d\to\mathbb{R}$$
We will show that in this situation our assumptions will hold for
$\xi$ as soon as $Z$ has an infinitesimal generator whose
coefficients satisfy Assumptions \ref{A1}, \ref{A2} and \ref{C},
allowing us to construct the Markovian Projection of $\xi_t$.

Consider a time-dependent integro-differential operator
$L=(L_t)_{t\in[0,T]}$ defined, for  $g\in\mathcal{C}_0^\infty(\mathbb{R}^d)$, by
  \begin{equation}\label{LZ.eq}
   \begin{split}
      L_tg(z)&=b_Z(t,z).\nabla g(z)+\sum_{i,j=1}^d\frac{(a_Z)_{ij}(t,x)}{2}\frac{\partial^2 g}{\partial x_i\partial x_j} (x)\\[0.1cm]
      &+\int_{\mathbb{R}^d}[g(z+\psi_Z(t,z,y)-g(z)-\psi_Z(t,y,z).\nabla g(z)]\nu_Z(y) dy,
   \end{split}
  \end{equation}
where  $b_Z:[0,T]\times \mathbb{R}^d\mapsto \mathbb{R}^d$, $a_Z:[0,T]\times \mathbb{R}^d\mapsto M_{d\times d}(\mathbb{R})$, and $\psi_Z:[0,T]\times \mathbb{R}^d\times \mathbb{R}^d$ are measurable functions and $\nu_Z$ is a L\'{e}vy density.\\[0.1cm]
If one assume that
\begin{equation}\label{chap.mim.psi}
\begin{split}
&\psi_Z(.,.,0)=0\quad \psi_Z(t,z,.)\,\mathrm{is\,a\,} \mathcal{C}^1(\mathbb{R}^d,\mathbb{R}^d)-{\rm diffeomorphism},\\[0.1cm]
& \forall t\in[0,T]\:\forall z\in\mathbb{R}^d\:\mathbb{E}\left[\int_0^t \int_{\{\|y\|\geq 1\}} \left(1\wedge \|\psi_Z(s,z,y)\|^2\right)\,\nu_Z(y)\,dy\,ds \right]<\infty,
\end{split}
\end{equation}
then applying Lemma \ref{lemma2}, \eqref{LZ.eq} rewrites, for  $g\in\mathcal{C}_0^\infty(\mathbb{R}^d)$:
  \begin{equation}\label{LZ.eq.bis}
   \begin{split}
      L_tg(x)&=b_Z(t,x).\nabla g(x)+\sum_{i,j=1}^d\frac{(a_Z)_{ij}(t,x)}{2}\frac{\partial^2 g}{\partial x_i\partial x_j} (x)\\[0.1cm]
      &+\int_{\mathbb{R}^d}[g(x+y)-g(x)-\,y.\nabla g(x)]m_Z(t,y,x)\,dy,
   \end{split}
  \end{equation}
where
\begin{equation}\label{mvspsi}
m_Z(t,y,x)= 1_{\{y\in\psi_Z(t,\mathbb{R}^d,x)\}}\left|{\rm det}\nabla \psi_Z\right|^{-1}(t,x,\psi_Z^{-1}(t,x,y))\,\nu_Z(\psi_Z^{-1}(t,x,y)).
\end{equation}
Throughout this section we shall assume that $(b_Z,a_Z,m_Z)$ satisfy Assumptions \ref{A1}, \ref{A2} and \ref{C}. Proposition \ref{th.well.posed} then implies that for any $Z_0\in\mathbb{R}^d$, the SDE,
\begin{equation}
  \begin{split}
 \forall t\in[0,T] \quad  Z_t&=Z_0+\int_0^t b_Z(u,Z_{u-})\,du+\int_0^t a_Z(u,Z_{u-})\,dW_u\\[0.1cm]
    &+\int_0^t\int \psi_Z(u,Z_{u-},y)\,\tilde{N}(du\ dy),
  \end{split}
  \end{equation}
admits a weak solution  $((Z_t)_{t\in[0,T]},\mathbb{Q}_{Z_0})$, unique in law, with $(W_t)$ an n-dimensional Brownian motion, $N$ a Poisson
  random measure on $[0,T]\times\mathbb{R}^d$ with compensator
  $\nu_Z(y)\,dy\,dt$, $\tilde{N}$ the associated
compensated random measure. Under $\mathbb{Q}_{Z_0}$, $Z$ is a   Markov process with infinitesimal generator $L$.

Consider now the process
\begin{equation}
  \xi_t=f(Z_t).
\end{equation}
The aim of this section is to build in an explicit manner the Markovian Projection of $\xi_t$ for a sufficiently large class of functions $f$.

Let us first rewrite $\xi_t$ in the form (\ref{classeJ}):
\begin{proposition}\label{semi.decomp.f.de.markov}
Let $f\in {\mathcal C}^2(\mathbb{R}^d,\mathbb{R})$ with bounded derivatives such that,
\begin{equation*}
\forall (z_1,\cdots,z_{d-1})\in \mathbb{R}^{d-1},\quad u\mapsto f(z_1,\ldots,z_{d-1}, u)\:\,\mathrm{is\,a\,} \mathcal{C}^1(\mathbb{R},\mathbb{R})-{\rm diffeomorphism}.
\end{equation*}
Assume that $(a_Z,m_Z)$ satisfy Assumption \ref{C}, then $\xi_t=f(Z_t)$ admits the following semimartingale decomposition:
\begin{equation*}
\xi_t=\xi_0+\int_0^t \beta_s\,ds+ \int_0^t  \delta_s \,dB_s +\int_0^t \int u \,\tilde{K}(ds\:du),
\end{equation*}
where
\begin{equation}
  \begin{cases}
    \beta_t&=\nabla f(Z_{t^-}).b_Z(t,Z_{t-})+\frac{1}{2}{\rm tr}\left[\nabla^2 f(Z_{t-}){}^ta_Z(t,Z_{t-})a_Z(t,Z_{t-})\right]\\[0.15cm]
    &+\int_{\mathbb{R}^d}\left(f(Z_{t^-}+\psi_Z(t,Z_{t-},y))-f(Z_{t^-})-\psi_Z(t,Z_{t-},y).\nabla f(Z_{t^-})\right)\,\nu_Z(y)\,dy,\\[0.15cm]
    \delta_t&=\|\nabla f(Z_{t-})a_Z(t,Z_{t-})\|,\\[0.15cm]
  \end{cases}
\end{equation}
$B$ is a real valued Brownian motion and $K$ is an integer-valued random measure on $[0,Tj]\times\mathbb{R}$ with compensator $k(t,Z_{t-},u)\,du\,dt$ defined for all $z\in\mathbb{R}^d$ and for any $u>0$ (and analogously for $u<0$) via:
\begin{equation}
k(t,z,[u,\infty[)=\int_{\mathbb{R}^d} 1_{\{f(z+\psi_Z(t,z,y))-f(z)\geq u\}}\,\nu_Z(y)\,dy.
\end{equation}
and $\tilde{K}$ its compensated random measure.
\end{proposition}
\begin{proof}
Applying It\^{o}'s formula to $\xi_t=f(Z_t)$ yields
\begin{eqnarray*}
  \xi_t&=&\xi_0+\int_0^t \nabla f(Z_{s^-}).b_Z(s,Z_{s-})\,ds+ \int_0^t \nabla f(Z_{s^-}).a_Z(s,Z_{s-}) dW_s\\[0.1cm]
  &+&\frac{1}{2}\int_0^t {\rm tr}\left[\nabla^2 f(Z_{s-}){}^ta_Z(s,Z_{s-})a_Z(s,Z_{s-})\right]\,ds+\int_0^t\nabla f(Z_{s^-}).\psi_Z(s,Z_{s-},y)\ \,\tilde{N}(ds\,dy)\\[0.1cm]
   &+&\int_0^t\int_{\mathbb{R}^d}\left(f(Z_{s^-}+\psi_Z(s,Z_{s-},y))-f(Z_{s^-})-\psi_Z(s,Z_{s-},y).\nabla f(Z_{s^-})\right)\,N(ds\,dy)\\[0.1cm]
&=&\xi_0+\int_0^t \Big[\nabla f(Z_{s^-}).b_Z(s,Z_{s-})+\frac{1}{2}{\rm tr}\left[\nabla^2 f(Z_{s-}){}^ta_Z(s,Z_{s-})a_Z(s,Z_{s-})\right]\\[0.1cm]
&&\quad\quad\quad\quad\quad+\int_{\mathbb{R}^d}\left(f(Z_{s^-}+\psi_Z(s,Z_{s-},y))-f(Z_{s^-})-\psi_Z(s,Z_{s-},y).\nabla f(Z_{s^-})\right)\,\nu_Z(y)\,dy\Big]\,ds\\[0.1cm]
    &+& \int_0^t \nabla f(Z_{s^-}).a_Z(s,Z_{s-}) dW_s+\int_0^t\int_{\mathbb{R}^d}\left(f(Z_{s^-}+\psi_Z(s,Z_{s-},y))-f(Z_{s^-})\right)\,\tilde{N}(ds\,dy).
\end{eqnarray*}
Given Assumption \ref{C}, either
\begin{equation}
\forall R>0\:\forall t\in [0,T] \quad \inf_{\|z\| \leq R}\,\inf_{x\in\mathbb{R}^d,\,\|x\|=1} {}^tx.a_Z(t,z).x>0,
\end{equation}
then $(B_t)_{t\in[0,T]}$ defined by
\[
dB_t=\frac{\nabla f(Z_{t^-}).a_Z(t,Z_{t-})W_t}{\|\nabla f(Z_{t-})a_Z(t,Z_{t-})\|},
\]
 is a continuous local martingale with  $[B]_t=t$ thus  a Brownian motion,
or $a_Z\equiv 0$, then $\xi$ is a pure-jump semimartingale. Define $\mathcal{K}_t$
\begin{equation*}
  \mathcal{K}_t=\int_0^t\int \Psi_Z(s,Z_{s-},y)\,\tilde{N}(ds\:dy),
\end{equation*}
with $\Psi_Z(t,z,y)=\psi_Z(t,z,\kappa_z(y))$ where
\begin{eqnarray*}
  \kappa_z(y):\mathbb{R}^d&\to&\mathbb{R}^d\\[0.1cm]
    y&\to&(y_1,\cdots,y_{d-1},f(z+y)-f(z)).
\end{eqnarray*}
Since for any $z\in\mathbb{R}^d$, $\left|{\rm det}\nabla_y \kappa_z\right|(y)=\left|\frac{\partial}{\partial y_{d}} f(z+y)\right|>0$, one can define
\begin{equation*}
\kappa_z^{-1}(y)=(y_1,\cdots,y_{d-1},F_z(y))\quad F_z(y):\mathbb{R}^d\to\mathbb{R} \quad f(z+(y_1,\cdots,y_{d-1},F_z(y)))-f(z)=y_d.
\end{equation*}
Considering $\phi$ the inverse of $\psi$ that is $\phi(t,\psi_Z(t,z,y),z)= y$, define
\begin{equation*}
\Phi(t,z,y)=\phi(t,z,\kappa_z^{-1}(y)).
\end{equation*}
$\Phi$ corresponds to the inverse of $\Psi_Z$ and $\Phi$ is differentiable on $\mathbb{R}^d$ with image $\mathbb{R}^d$. Now, define
\begin{equation*}
\begin{split}
m(t,z,y)&=\left|{\rm det}\nabla_y\Phi(t,z,y)\right|\,\nu_Z(\Phi(t,z,y))\\[0.1cm]
&=\left|{\rm det}\nabla_y \phi(t,z,\kappa_z^{-1}(y))\right| \,\left|\frac{\partial f}{\partial y_{d}}(z+\kappa_z^{-1}(y))\right|^{-1}\,\nu_Z(\phi(t,z,\kappa_z^{-1}(y))).
\end{split}
\end{equation*}
One observes that
\begin{eqnarray*}
&&\int_0^t \int_{\|y\|>1} \left(1\wedge \|\Psi_Z(s,z,y)\|^2\right)\,\nu_Z(y)\,dy\,ds\ =\\[0.1cm]
&&\int_0^t \int_{\|y\|>1} \left(1\wedge \left(\psi^1(s,z,y)^2+\cdots+\psi^{d-1}(s,z,y)^2+(f(z+\psi_Z(s,z,y))-f(z))^2\right)\right)\,\nu_Z(y)\,dy\,ds\\[0.1cm]
&\leq&\int_0^t \int_{\|y\|>1}  \left(1\wedge \left(\psi^1(s,z,y)^2+\cdots+\psi^{d-1}(s,z,y)^2+\|\nabla f\|^2\|\psi_Z(s,z,y)\|^2\right)\right)\,\nu_Z(y)\,dy\,ds\\[0.1cm]
&\leq&\int_0^t \int_{\|y\|>1} \left( 1\wedge (2\vee \|\nabla f\|^2)\|\psi_Z(s,z,y)\|^2\right)\,\nu_Z(y)\,dy\,ds.\\[0.1cm]
\end{eqnarray*}
Given the condition (\ref{chap.mim.psi}), one may apply Lemma \ref{lemma2} and express $\mathcal{K}_t$ as
$\mathcal{K}_t=\int_0^t\int y\tilde{M}(ds\:dy)$
where $\tilde{M}$ is a compensated integer-valued random measure on $[0,T]\times\mathbb{R}^d$  with compensator $m(t,Z_{t-},y)\,dy\,dt$.\\[0.1cm]
Extracting the $d$-th component of $\mathcal{K}_t$, one obtains the semimartingale decomposition of $\xi_t$ on $[0,T]$
\begin{equation*}
\xi_t=\xi_0+\int_0^t \beta_s\,ds+ \int_0^t  \delta_s \,dB_s +\int_0^t \int u \,\tilde{K}(ds\:du),
\end{equation*}
where
\begin{equation*}
  \begin{cases}
    \beta_t&=\nabla f(Z_{t^-}).b_Z(t,Z_{t-})+\frac{1}{2}{\rm tr}\left[\nabla^2 f(Z_{t-}){}^ta_Z(t,Z_{t-})a_Z(t,Z_{t-})\right]\\[0.15cm]
    &+\int_{\mathbb{R}^d}\left(f(Z_{t^-}+\psi_Z(t,Z_{t-},y))-f(Z_{t^-})-\psi_Z(t,Z_{t-},y).\nabla f(Z_{t^-})\right)\,\nu_Z(y)\,dy,\\[0.15cm]
    \delta_t&=\|\nabla f(Z_{t-})a_Z(t,Z_{t-})\|,\\[0.15cm]
  \end{cases}
\end{equation*}
and $K$ is an integer-valued random measure on $[0,T]\times\mathbb{R}$ with compensator $k(t,Z_{t-},u)\,du\,dt$ defined for all $z\in\mathbb{R}^d$ via
\begin{equation*}
\begin{split}
k(t,z,u)&= \int_{\mathbb{R}^{d-1}} m(t,(y_1,\cdots,y_{d-1},u),z)\,dy_{1}\cdots\,dy_{d-1}\\[0.1cm]
&=\int_{\mathbb{R}^{d-1}} \left|{\rm det}\nabla_y\Phi(t,z,(y_1,\cdots,y_{d-1},u))\right|\,\nu_Z(\Phi(t,z,(y_1,\cdots,y_{d-1},u)))\,dy_{1}\cdots\,dy_{d-1},
\end{split}
\end{equation*}
and $\tilde{K}$ its compensated random measure. In particular for any $u>0$ (and analogously for $u<0$), 
\begin{equation*}
k(t,z,[u,\infty[)=\int_{\mathbb{R}^d} 1_{\{f(z+\psi_Z(t,z,y))-f(z)\geq u\}}\,\nu_Z(y)\,dy.
\end{equation*}
\end{proof}

Given the semimartingale decomposition of $\xi_t$ in the form
(\ref{classeJ}), we may now construct the Markovian projection of
$\xi$  as follows.
\begin{theorem}\label{th.fonction.markov.process}
Assume that
\begin{itemize}
  \item the coefficients $(b_Z,a_Z,m_Z)$ satisfy Assumptions
\ref{A1}, \ref{A2} and \ref{C},
  \item the Markov process $Z$  has a transition density $q_{t}(.)$ which is
continuous on $\mathbb{R}^d$ uniformly in $t\in[0,T]$, and $t
\mapsto q_{t}(z)$ is right-continuous on $[0,T[$, uniformly in
$z\in\mathbb{R}^d$.
  \item $f\in C^2_b(\mathbb{R}^d,\mathbb{R})$ such that
\begin{equation*}
\forall (z_1,\cdots,z_{d-1})\in \mathbb{R}^{d-1},\quad u\mapsto f(z_1,\ldots,z_{d-1}, u)\:\,\mathrm{is\,a\,} \mathcal{C}^1(\mathbb{R},\mathbb{R})-{\rm diffeomorphism}.
\end{equation*}
\end{itemize}
Define,  for  $w\in \mathbb{R},t\in[0,T],$
\begin{equation}\label{new.paramet.markov}
\begin{split}
b(t,w)&=\frac{1}{c(w)}\,\int_{\mathbb{R}^{d-1}} \Big[\nabla f(.).b_Z(t,.)+\frac{1}{2}{\rm tr}\left[\nabla^2 f(.){}^ta_Z(t,.)a_Z(t,.)\right]\\[0.1cm]
 &\quad\quad+\int_{\mathbb{R}^d}\left(f(.+\psi_Z(t,.,y))-f(.)-\psi_Z(t,.,y).\nabla f(.)\right)\,\nu_Z(y)\,dy,\Big](z_1,\cdots,z_{d-1},w)\\[0.1cm]
&\quad\quad\quad\quad\quad\quad\quad\quad\quad\times\frac{q_t(z_1,\cdots,z_{d-1},F(z_1,\cdots,z_{d-1},w))}{\left|\frac{\partial f}{\partial z_{d}}\right|(z_1,\cdots,z_{d-1},F(z_1,\cdots,z_{d-1},w))},\\[0.cm]
 \sigma(t,w)&=\frac{1}{\sqrt{c(w)}}\,\Big[\int_{\mathbb{R}^{d-1}} \|\nabla f(.)a_Z(t,.)\|^2(z_1,\cdots,z_{d-1},w)\\[0.1cm]
&\quad\quad\quad\quad\quad\quad\quad\quad\quad\,\times\frac{q_t(z_1,\cdots,z_{d-1},F(z_1,\cdots,z_{d-1},w))}{\left|\frac{\partial f}{\partial z_{d}}\right|(z_1,\cdots,z_{d-1},F(z_1,\cdots,z_{d-1},w))}\Big]^{1/2},\\[0.1cm]
j(t,[u,\infty[,w)&=\frac{1}{c(w)}\,\int_{\mathbb{R}^{d-1}}\left(\int_{\mathbb{R}^d} 1_{\{f(.+\psi_Z(t,.,y))-f(.)\geq u\}}(z_1,\cdots,z_{d-1},w)\,\nu_Z(y)\,dy\right)\\[0.1cm]
&\quad\quad\quad\quad\quad\quad\quad\,\times\frac{q_t(z_1,\cdots,z_{d-1},F(z_1,\cdots,z_{d-1},w))}{\left|\frac{\partial f}{\partial z_{d}}(z_1,\cdots,z_{d-1},F(z_1,\cdots,z_{d-1},w))\right|},
\end{split}
\end{equation}
for $u>0$ (and analogously for $u<0$),
with
\begin{equation*}
c(w)=\int_{\mathbb{R}^{d-1}} dz_1...dz_{d-1}\frac{q_t(z_1,\cdots,z_{d-1},F(z_1,\cdots,z_{d-1},w))}{\left|\frac{\partial f}{\partial z_{d}}(z_1,\cdots,z_{d-1},F(z_1,\cdots,z_{d-1},w))\right|}.
\end{equation*}
Then the stochastic differential equation
\begin{equation}
  \begin{split}
    X_t&=\xi_0+\int_0^t b(s,X_s)\,ds+\int_0^t \sigma(s,X_s)\,dB_s\\[0.1cm]
    &+\int_0^t\int_{\|y\|\leq 1}y \,\tilde{J}(ds\, dy)+ \int_0^t\int_{\|y\|> 1}y \,J(ds\, dy),
  \end{split}
  \end{equation}
where  $(B_t)$ is a Brownian motion, $J$ is an integer-valued
  random measure on $[0,T]\times\mathbb{R}$ with compensator
  $j(t,du,X_{t-})\,dt$, $\tilde{J}=J-j$, admits a weak solution  $((X_t)_{t\in[0,T]},\mathbb{Q}_{\xi_0})$, unique in law,
whose  marginal distributions  mimick those of $\xi$:
  $$\forall t\in[0,T],\quad X_t\overset{\underset{\mathrm{d}}{}}{=}\xi_t. $$
Under $\mathbb{Q}_{\xi_0}$, $X$ is a   Markov process with infinitesimal
generator $\mathcal{L}$ given by
  \begin{equation*}
   \begin{split}
\forall g\in\mathcal{C}_0^\infty(\mathbb{R})\quad \mathcal{L}_tg(w)&=b(t,w)g'(w)+\frac{\sigma^2(t,w)}{2}\,g''(w)\\[0.1cm]
      &+\int_{\mathbb{R}^d}[g(w+u)-g(w)-ug'(w)]j(t,du,w).
   \end{split}
  \end{equation*}
\end{theorem}
Before proving Theorem \ref{th.fonction.markov.process}, we start
with an useful Lemma, which will be of importance.
\begin{lemma}\label{lemma.cond.law.good.case}
Let $Z$ be a $\mathbb{R}^d$-valued random variable with density $q(z)$ and $f\in\mathcal{C}^1(\mathbb{R}^d,\mathbb{R})$ such that
\begin{equation*}
\forall (z_1,\cdots,z_{d-1})\in \mathbb{R}^{d-1},\quad u\mapsto f(z_1,\ldots,z_{d-1}, u)\:\,\mathrm{is\,a\,} \mathcal{C}^1(\mathbb{R},\mathbb{R})-{\rm diffeomorphism}.
\end{equation*}
Define the function $F:\mathbb{R}^{d}\to\mathbb{R}$ such that
 $ f(z_1,\cdots,z_{d-1},F(z))=z_{d}$.
Then for any measurable function $g:\mathbb{R}^d\to\mathbb{R}$  such
that $\mathbb{E}\left[|g(Z)|\right]<\infty$ and  any $w\in
\mathbb{R},$
\begin{eqnarray*}
&&\mathbb{E}\left[g(Z)|f(Z)=w\right]\\[0.1cm]
&=&\frac{1}{c(w)}\,\int_{\mathbb{R}^{d-1}} dz_1...dz_{d-1} g(z_1,\cdots,z_{d-1},w)\frac{q(z_1,\cdots,z_{d-1},F(z_1,\cdots,z_{d-1},w))}{\left|\frac{\partial f}{\partial z_{d}}(z_1,\cdots,z_{d-1},F(z_1,\cdots,z_{d-1},w))\right|},
\end{eqnarray*}
with
\begin{equation*}
c(w)=\int_{\mathbb{R}^{d-1}} dz_1...dz_{d-1}\frac{q(z_1,\cdots,z_{d-1},F(z_1,\cdots,z_{d-1},w))}{\left|\frac{\partial f}{\partial z_{d}}(z_1,\cdots,z_{d-1},F(z_1,\cdots,z_{d-1},w))\right|}.
\end{equation*}
\end{lemma}
\begin{proof}
Consider the d-dimensional random variable $\kappa(Z)$, where $\kappa:\mathbb{R}^d\mapsto\mathbb{R}^d $ is given by
\begin{equation*}
\kappa(z)=\left(z_1,\cdots,z_{d-1},f(z)\right).
\end{equation*}
\begin{equation*}
(\nabla_z\kappa)=\left(
\begin{array}{cccc}
1 & 0 &0 & 0\\[0.1cm]
\vdots & \ddots & \vdots& \vdots\\[0.1cm]
0& \cdots & 1 & 0\\[0.1cm]
\frac{\partial f}{\partial z_1}&\cdots&\frac{\partial f}{\partial z_{d-1}}&\frac{\partial f}{\partial z_d}
\end{array}
\right),
\end{equation*}
so t $|\mathrm{det}(\nabla_z\kappa)|(z)=\left|\frac{\partial f}{\partial z_d}(z)\right|>0$. Hence $\kappa$ is a $\mathcal{C}^1(\mathbb{R}^d,\mathbb{R}^d)$-diffeomorphism with inverse $\kappa^{-1}$.
\begin{equation*}
\kappa(\kappa^{-1}(z))=(\kappa^{-1}_1(z),\cdots,\kappa_{d-1}^{-1}(z),f(\kappa_1^{-1}(z),\cdots,\kappa_d^{-1}(z))=z.
\end{equation*}
For $1\leq i \leq d-1$, $\kappa_i^{-1}(z)=z_i$ and $f(z_1,\cdots,z_{d-1},\kappa_d^{-1}(z))=z_d$ that is $\kappa_d^{-1}(z)=F(z)$. Hence
\begin{equation*}
\kappa^{-1}(z_1,\cdots,z_d)=(z_1,\cdots,z_{d-1},F(z)).
\end{equation*}
Define $q_{\kappa}(z)\,dz$ the inverse image of the measure $q(z)\,dz$ under the partial map $\kappa$ by
\begin{eqnarray*}
q_{\kappa}(z)&=&1_{\{\kappa(\mathbb{R}^d)\}}(z)\,|\mathrm{det}(\nabla_z \kappa^{-1})|(z)\,q(\kappa^{-1}(z))\\[0.1cm]
&=&1_{\{\kappa(\mathbb{R}^d)\}}(z)\left|\frac{\partial f}{\partial z_{d}}\right|^{-1}(z_1,\cdots,z_{d-1},F(z))\,q(z_1,\cdots,z_{d-1},F(z)).\\[0.1cm]
\end{eqnarray*}
$q_{\kappa}(z)$ is the density of $\kappa(Z)$. So, for any $w\in f(\mathbb{R}^d)=\mathbb{R}$,
\begin{eqnarray*}
&&\mathbb{E}\left[g(Z)|f(Z)=w\right]\\[0.1cm]
&=&\int_{\mathbb{R}^{d-1}}\mathbb{E}\left[g(Z)|\kappa(Z)=(z_1,\cdots,z_{d-1},w)\right]\,dz_1,\cdots,dz_{d-1}\\[0.1cm]
&=&\frac{1}{c(w)}\int_{\mathbb{R}^{d-1}} dz_1...dz_{d-1}g(z_1,\cdots,z_{d-1},w)\frac{q(z_1,\cdots,z_{d-1},F(z_1,\cdots,z_{d-1},w))}{\left|\frac{\partial f}{\partial z_{d}}\right|(z_1,\cdots,z_{d-1},F(z_1,\cdots,z_{d-1},w))},
\end{eqnarray*}
with
\begin{equation*}
c(w)=\int_{\mathbb{R}^{d-1}} dz_1...dz_{d-1}\frac{q(z_1,\cdots,z_{d-1},F(z_1,\cdots,z_{d-1},w))}{\left|\frac{\partial f}{\partial z_{d}}(z_1,\cdots,z_{d-1},F(z_1,\cdots,z_{d-1},w))\right|}.
\end{equation*}
\end{proof}

\begin{proof}[Proof of Theorem \ref{th.fonction.markov.process}]
Let us show that if $(b_Z,a_Z,m_Z)$ satisfy Assumptions \ref{A1}, \ref{A2} and \ref{C} then the triplet $(\delta_t,\beta_t,k(t,Z_{t-},u))$ satisfies the assumptions of Theorem \ref{th1}. Then given Proposition \ref{semi.decomp.f.de.markov}, one may build in an explicit manner the Markovian Projection of $\xi_t$.

First, note that $\beta_t$ and $\delta_t$ satisfy Assumption \ref{H1} since $b_Z(t,z)$ and $a_Z(t,z)$ satisfy Assumption \ref{A1} and $\nabla f$ and $\nabla^2 f$ are bounded.

One observes that if $m_Z$ satisfies Assumption \ref{A1}, then the equality (\ref{mvspsi}) implies that $\psi_Z$ and $\nu_Z$ satisfies:
\begin{equation}\label{H3bb}
 \exists K_2>0 \,\forall t\in[0,T]\:\forall z\in\mathbb{R}^d\:\int_0^t \int_{\{\|y\|\geq 1\}} \left(1\wedge \|\psi_Z(s,z,y)\|^2\right)\,\nu_Z(y)\,dy\,ds <K_2.
\end{equation}
Hence,
\begin{eqnarray*}
\int_0^t\int \left(1\wedge \|u\|^2\right)\,k(s,Z_{s-},u)\,du\,ds
&=&\int_0^t\int \left(1\wedge |f(Z_{s^-}+\psi_Z(s,Z_{s-},y))-f(Z_{s^-})|^2\right)\,\nu_Z(y)\,dy\,ds\\[0.1cm]
&\leq&\int_0^t\int \left(1\wedge \|\nabla f\|^2 \|\psi_Z(s,Z_{s-},u)\|^2\right) \,\nu_Z(y)\,dy\,ds,
\end{eqnarray*}
is bounded and $k$ satisfies Assumption \ref{H2}.\\[0.1cm]
As argued before, one sees that if $a_Z$ is non-degenerate then $\delta_t$ is. In the case $\delta_t\equiv 0$, for $t\in [0,T[,R>0,z\in B(0,R)$ and $g\in\mathcal{C}_0^0(\mathbb{R})\geq 0$, denoting $C$ and $K_{T}>0$ the constants in
Assumption \ref{C},
\begin{eqnarray*}
&&k(t,z,u)\\[0.1cm]
&=&\int_{\mathbb{R}^{d-1}} \left|{\rm det}\nabla_y\Phi(t,z,(y_1,\cdots,y_{d-1},u))\right|\,\nu_Z(\Phi(t,z,(y_1,\cdots,y_{d-1},u)))\,dy_1\,\cdots\, dy_{d-1}\\[0.1cm]
&=&\int_{\mathbb{R}^{d-1}} \left|{\rm det}\nabla_y \phi(t,z,\kappa_z^{-1}(y_1,\cdots,y_{d-1},u))\right|\,\left|\frac{\partial f}{\partial y_{d}}(z+\kappa_z^{-1}(y_1,\cdots,y_{d-1},u))\right|^{-1}\\[0.1cm]
&&\quad\quad\quad\quad\quad\quad \quad\quad\quad \quad\quad\quad \nu_Z(\phi(t,z,\kappa_z^{-1}(y_1,\cdots,y_d,u)))\,dy_1\,\cdots\, dy_{d-1}\\[0.1cm]
&\geq& \int_{\mathbb{R}^{d-1}} \left|\frac{\partial f}{\partial y_{d}}(z+\kappa_z^{-1}(y_1,\cdots,y_{d-1},u))\right|^{-1} \frac{C}{\|\kappa_z^{-1}(y_1,\cdots,y_{d-1},u)\|^{d+\beta}}\,dy_1\,\cdots\, dy_{d-1}\\[0.1cm]
&=&\int_{\mathbb{R}^{d-1}}\frac{C}{\|(y_1,\cdots,y_{d-1},u)\|^{d+\beta}}\,dy_1\,\cdots\,dy_{d-1}\\[0.1cm]
&=&\frac{1}{|u|^{d+\beta}}\int_{\mathbb{R}^{d-1}}\frac{C}{\|(y_1/u,\cdots,y_{d-1}/u,1)\|^{d+\beta}}\,dy_1\,\cdots\,dy_{d-1}= C'\frac{1}{|u|^{1+\beta}},
\end{eqnarray*}
with $C'=\int_{\mathbb{R}^{d-1}} C\,\|(w_1,\cdots,w_{d-1},1)\|^{-1}\,dw_1\,\cdots\,dw_{d-1}$.\\[0.1cm]
Similarly
\begin{eqnarray*}
&& \int \left(1\wedge |u|^{\beta}\right)\,\left(k(t,z,u)-\frac{C'}{|u|^{1+\beta}}\right)\,du\\[0.1cm]
&=&  \int \left(1\wedge |u|^{\beta}\right)\,\int_{\mathbb{R}^{d-1}} \,\left|\frac{\partial f}{\partial y_{d}}(z+\kappa_z^{-1}(y_1,\cdots,y_{d-1},u))\right|^{-1} \\[0.1cm]
&&\quad \quad\Big[ \left|{\rm det}\nabla_y \phi(t,z,\kappa_z^{-1}(y_1,\cdots,y_{d-1},u))\right|\,\nu_Z(\phi(t,z,\kappa_z^{-1}(y_1,\cdots,y_{d-1},u)))\\[0.1cm]
&&\quad\quad \quad\quad\quad \quad-\frac{C}{\|\kappa_z^{-1}(y_1,\cdots,y_{d-1},u)\|^{d+\beta}}\Big]\,dy_1\,\cdots\, dy_{d-1}\,du\\[0.1cm]
&=& \int_{\mathbb{R}^{d}} \left(1\wedge |f(z+(y_1,\cdots,y_{d-1},u))-f(z)|^{\beta}\right)\\[0.1cm]
&&\quad\quad \quad\quad\Big(\left|{\rm det}\nabla_y \phi(t,z,(y_1,\cdots,y_{d-1},u))\right|\,\nu_Z(\phi(t,z,(y_1,\cdots,y_{d-1},u)))\\[0.1cm]
&&\quad\quad \quad\quad\quad \quad-\frac{C}{\|(y_1,\cdots,y_{d-1},u)\|^{d+\beta}}\Big)\,dy_1\,\cdots\, dy_{d-1}\,du\\[0.1cm]
&\leq& \int_{\mathbb{R}^{d}} \left(1\wedge \|\nabla f\|\|(y_1,\cdots,y_{d-1},u)\|^{\beta}\right)\nonumber\\[0.1cm]
&&\Big(\left|{\rm det}\nabla_y \phi(t,z,(y_1,\cdots,y_{d-1},u))\right|\,\nu_Z(\phi(t,z,(y_1,\cdots,y_{d-1},u)))\\[0.1cm]
&&-\frac{C}{\|(y_1,\cdots,y_{d-1},u)\|^{d+\beta}}\Big)\,dy_1\,\cdots\,
dy_{d-1}\,du
\end{eqnarray*}
is also bounded.
Similar arguments would show that
\begin{eqnarray*}
 && \lim_{\epsilon\to 0} \int_{|u|\leq \epsilon} |u|^{\beta}\,\left(k(t,Z_{t-},u)-\frac{C\,}{|u|^{1+\beta}}\right)\,du=0\:\mathrm{a.s.} \\[0.1cm]
          &\mathrm{and}& \lim_{R\to\infty} \int_0^T  k\left(t,Z_{t-},\{|u|\geq R\}\right)\,dt=0\:\mathrm{a.s.} ,
\end{eqnarray*}
since this essentially hinges on the fact that $f$ has bounded derivatives.

Applying Lemma \ref{lemma.cond.law.good.case}, one can compute explicitly
the conditional expectations in \eqref{new_param}. For example,
\begin{eqnarray*}
b(t,w)&=&E[\beta_t|\xi_{t-}=w]=\int_{\mathbb{R}^{d-1}} \Big[\nabla f(.).b_Z(t,.)+\frac{1}{2}{\rm tr}\left[\nabla^2 f(.){}^ta_Z(t,.)a_Z(t,.)\right]\\[0.1cm]
 &&+\int_{\mathbb{R}^d}\left(f(.+\psi_Z(t,.,y))-f(.)-\psi_Z(t,.,y).\nabla f(.)\right)\,\nu_Z(y)\,dy,\Big](z_1,\cdots,z_{d-1},w)\\[0.1cm]
&&\quad\quad\quad\quad\times\frac{q_t(z_1,\cdots,z_{d-1},F(z_1,\cdots,z_{d-1},w))}{|\frac{\partial f}{\partial z_{d}}(z_1,\cdots,z_{d-1},F(z_1,\cdots,z_{d-1},w))|}.
\end{eqnarray*}
with $F:\mathbb{R}^d\to\mathbb{R}$ defined by
$f(z_1,\cdots,z_{d-1},F(z))=z_{d}$. Furthermore, $f$ is $C^2$ with
bounded derivatives, $(b_Z,a_Z,\nu_Z)$ satisfy Assumption \ref{A1}.
Since $z\in\mathbb{R}^d\to q_t(z)$ is continuous in $z$ uniformly in
$t\in[0,T]$ and $t\in[0,T[\to q_t(z)$ is right-continuous in $t$ uniformly in
$z\in\mathbb{R}^d$, the same properties hold for $b$. Proceeding similarly, one can show that
that Assumption \ref{A2} holds for $\sigma$ and $j$ so Theorem
\ref{th1} may be applied to yield the result.
\end{proof}




\subsection{Time changed  L\'evy processes}\label{timechange.sec}
Models based on
time--changed L\'evy processes  have been the focus of much recent
work, especially in mathematical finance \cite{carrgeman03}. Let
$L_t$ be a L\'evy process, $(b,\sigma^2,\nu)$ be its
characteristic triplet on $(\Omega,({\cal F}_t)_{t\geq 0},\mathbb{P})$, $N$  the Poisson random measure representing the jumps of $L$ and $(\theta_t)_{t\geq 0}$  a locally bounded, strictly positive $\mathcal{F}_t$-adapted  cadlag process.
 The process
\begin{equation*}
\xi_t= \xi_0 + L_{\Theta_t} \qquad \Theta_t=\int_0^t \theta_s ds.
\end{equation*}
 is called a time-changed L\'evy process where $\theta_t$ is  interpreted as the rate of time change.
\begin{theorem}[Markovian projection of time-changed L\'evy processes]\label{th.time.changed.levy}
Assume that $(\theta_t)_{t\geq 0}$ is bounded from above and away from zero:
\begin{equation}\label{chap1.encadrement}
\begin{split}
 \exists K,\epsilon>0, \forall t\in[0,T],\:\: K \geq\theta_t\geq \epsilon\qquad \mathrm{a.s.}
\end{split}
\end{equation}
and that there exists  $\alpha:[0,T]\times \mathbb{R}\mapsto \mathbb{R}$ such that
$$ \forall t\in[0,T],\quad\forall z\in\mathbb{R},\quad\quad\alpha(t,z)=E[\theta_t| \xi_{t-}=z],$$
where $\alpha(t,.)$ is continuous  on
$\mathbb{R}$, uniformly in $t\in [0,T]$ and, for all $z$ in $\mathbb{R}$, $\alpha(.,z)$ is right-continuous in $t$ on $[0,T[$.
\begin{eqnarray*}
  &\mathrm{If\  either}& (i)\quad \sigma>0 \\[0.1cm]
  &{\rm or} & (ii) \quad \sigma\equiv 0\quad\mathrm{and}\:\exists\beta\in]0,2[, c,K'>0,\,\mathrm{and}\:\mathrm{a\,measure}\\[0.1cm]
&& \nu^\beta(dy)\:\mathrm{such\,that} \quad \nu(dy)=\nu^\beta(dy)+\frac{c}{|y|^{1+\beta}}\,dy,  \\[0.1cm]
  &&\quad \int \left(1\wedge |y|^{\beta}\right)\,\nu^\beta(dy) \leq K',\quad \lim_{\epsilon\to 0} \int_{|y|\leq \epsilon} |y|^{\beta}\,\nu^\beta(dy)=0,
\end{eqnarray*}
then
\begin{itemize}\item $(\xi_t)$ has the same marginals as $(X_t)$ on $[0,T]$, defined as the weak
solution of
\begin{equation*}
  \begin{split}
    X_t = \xi_0+\int_0^t \sigma  \sqrt{\alpha(s,X_{s-})} dB_s + \int_0^t b \alpha(s,X_{s-}) ds  \\[0.1cm]
    + \int_0^t\int_{|z|\leq 1} z \tilde{J}(ds\,dz) + \int_0^t\int_{|z|>1} z J(ds\,dz),
 \end{split}
\end{equation*}
where $B_t$ is a real-valued brownian motion, $J$ is an integer-valued random measure on $[0,T]\times\mathbb{R}$ with compensator
$\alpha(t,X_{t-})\,\nu(dy)\,dt$.
\item The marginal distribution $ p_{t}$ of $\xi_t$ is the unique solution of the forward equation:
\begin{equation*}\frac{\partial p_{t}}{\partial
  t}= \mathcal{L}^{\star}_t.\,p_{t},
\end{equation*}
where, $\mathcal{L}^*_t$ is given by
\begin{eqnarray}
  \mathcal{L}^{\star}_t g(x) &=&-b \frac{\partial}{\partial x}\left[\alpha(t,x)g(x)\right]+\frac{\sigma^2}{2} \frac{\partial^2}{\partial x^2}[\alpha^2(t,x)g(x)]\nonumber\\[0.1cm]
&+& \int_{\mathbb{R}^d}\nu(dz)\left[g(x-z)\alpha(t,x-z)-g(x)\alpha(t,x)-1_{\|z\|\leq 1}z.\frac{\partial}{\partial x}[g(x)\alpha(t,x)]\right],\nonumber
\end{eqnarray}
with the given initial condition $p_0(dy)=\mu_0(dy)$ where $\mu_0$ denotes the law of $\xi_0$.
\end{itemize}
\end{theorem}
\begin{proof}
Consider the L\'evy-Ito decomposition of $L$:
 $$ L_t=bt+\sigma
W_t+  \int_0^t\int_{|z|\leq 1} z \tilde{N}(ds dz) +
\int_0^t\int_{|z|>1} z N(ds dz).$$
Then $\xi$ rewrites
\begin{equation*}
  \begin{split}
    \xi_t &= \xi_0+ \sigma  W(\Theta_t) +  b \Theta_t   \\[0.1cm]
    & +\int_0^{\Theta_t}\int_{|z|\leq 1} z \tilde{N}(ds\,dz) + \int_0^{\theta_t}\int_{|z|>1} z N(ds\,dz).
 \end{split}
 \end{equation*}
 $W(\Theta_t)$ is a continuous  martingale  starting from $0$, with quadratic variation $\Theta_t=\int_0^t \theta_s ds$. Hence, there exists $Z_t$ a Brownian motion, such that
$$W(\Theta_t)\mathop{=}^d \int_0^{t} \sqrt{\theta_s} dZ_s.$$
 Hence $\xi_t$ is the weak solution of :
\begin{equation*}
  \begin{split}
    \xi_t &= \xi_0+ \int_0^t \sigma \sqrt{\theta_s}\,dZ_s +  \int_0^t b\theta_s\,ds   \\[0.1cm]
    & +\int_0^{t}\int_{|z|\leq 1} z\theta_s \,\tilde{N}(ds\,dz) + \int_0^{t}\int_{|z|>1} z\theta_s\, N(ds\,dz).
 \end{split}
 \end{equation*}
Using the notations of Theorem \ref{th1},
\begin{equation*}
  \beta_t=b\,\theta_t,\quad
    \delta_t=\sigma\,\sqrt{\theta_t},\quad m(t,dy)=\theta_t\,\nu(dy).
\end{equation*}
Given the conditions (\ref{chap1.encadrement}), one simply observes that
\begin{equation*}
\forall (t,z)\in[0,T]\times\mathbb{R}, \quad \epsilon\leq \alpha(t,z)\leq K.
\end{equation*}
Hence Assumptions \ref{H1}, \ref{H2} and \ref{H3} hold for $(\beta,\delta,m)$. Furthermore,
\begin{equation*}
    \begin{split}
      b(t,.)&=\mathbb{E}\left[\beta_t|\xi_{t^-}=.\right]=b\,\alpha(t,.),\\[0.1cm]
      \sigma(t,.)&=\mathbb{E}\left[\delta_t^2|\xi_{t^-}=.\right]^{1/2}=\sigma\,\sqrt{\alpha(t,.)},\\[0.1cm]
      n(t,B,.)&=\mathbb{E}\left[m(t,B)|\xi_{t^-}=.\right]=\alpha(t,.)\nu(B),
    \end{split}
  \end{equation*}
are all continuous on $\mathbb{R}$ uniformly in $t$ on $[0,T]$ and for all $z\in\mathbb{R}$, $\alpha(.,z)$ is right-continuous on $[0,T[$. One may apply Theorem \ref{th1} yielding the result.
\end{proof}

 The impact of   the random time change on the marginals can be captured by making the characteristics state dependent
$$ (\ b \alpha(t,X_{t-}),\sigma^2 \alpha(t,X_{t-}),\alpha(t,X_{t-}) \nu\ )  $$
 by introducing the {\it same} adjustment factor $\alpha(t,X_{t-})$  to the drift, diffusion coefficient and L\'evy measure.
In particular if $\alpha(t,x)$ is affine in $x$ we get an affine process \cite{duffie2000taa} where the affine dependence
of the characteristics  with respect to the state  are restricted to be colinear, which is rather restrictive.
This remark shows that time-changed L\'evy processes, which in principle allow for a wide variety of choices for $\theta$ and $L$,
may not be as flexible as apparently simpler affine models when it comes to reproducing marginal distributions.

\def\polhk#1{\setbox0=\hbox{#1}{\ooalign{\hidewidth
  \lower1.5ex\hbox{`}\hidewidth\crcr\unhbox0}}}

\end{document}